\documentclass{amsart}
\usepackage[latin1]{inputenc}
\usepackage[english]{babel}
\usepackage[T1]{fontenc}
\usepackage{amsmath}
\usepackage{amssymb}
\usepackage{amsthm}
\usepackage[margin=3cm]{geometry}
\usepackage{array,booktabs}
\usepackage{graphicx}
\usepackage{pagecolor}
\usepackage{lmodern}
\usepackage{tcolorbox}
\usepackage{quoting}
\usepackage{enumerate}
\usepackage{mathtools}
\usepackage{csquotes}
\usepackage{cite}
\usepackage{tikz}
\usepackage{tikz-cd}
\usetikzlibrary{matrix,arrows,decorations.pathmorphing}
\usepackage{wrapfig}
\renewcommand{\Re}{\operatorname{Re}}
\renewcommand{\Im}{\operatorname{Im}}
\def\R{\ensuremath\mathbb{R}}
\def\C{\ensuremath\mathbb{C}}

\def\Z{\ensuremath\mathbb{Z}}

\def\wide{\ensuremath\mathbf{Wide}}

\usepackage{hyperref}
\usepackage[final]{pdfpages}
\usepackage{verbatim}
\newtheorem{thm}{Theorem}[section]

\newtheorem{cor}[thm]{Corollary}
\newtheorem{lemma}[thm]{Lemma}
\newtheorem{prop}[thm]{Proposition}
\theoremstyle{remark}
\newtheorem{remark}[thm]{Remark}

\def\eps{\ensuremath\varepsilon}
\def\0{\emptyset}

\def\GL{\mathrm{GL}}
\def\modulo{\text{ \rm mod }}

\numberwithin{equation}{section}

\numberwithin{equation}{section}
\begin{document}
\title[Wide moments II]{Wide moments of $L$-functions II: Dirichlet $L$-functions}
\author{Asbj\o rn Christian Nordentoft}

\address{Universit\'{e} Paris-Saclay, Laboratoire de Math\'{e}matiques d'Orsay, 307, rue Michel Magat, 91405 Orsay Cedex, France}

\email{\href{mailto:acnordentoft@outlook.com}{acnordentoft@outlook.com}}

\date{\today}
\maketitle
\begin{abstract}
We study \emph{wide moments} of Dirichlet $L$-functions using analytic properties of the Lerch zeta function. Among other things we obtain an asymptotic expansion of wide moments of Dirichlet $L$-functions (with arbitrary twists) extending results of Heath--Brown. We also give applications to non-vanishing. 
   \end{abstract}
\section{Introduction}
Understanding automorphic $L$-functions at the central point is an important and notoriously difficult problem in number theory. In this paper we deal with the case of Dirichlet $L$-function and in particular the asymptotic evaluation of certain moments as the moduli tends to infinity. The first result in this direction was obtained by Paley \cite{Paley31} who proved 
$$ \sum_{\chi\text{ primitive}\modulo q} |L(\chi,1/2)|^2=\frac{\varphi(q)\varphi^*(q)}{q}\log q(1+o(1)) , \quad q\rightarrow \infty,$$
where $\varphi(q),\varphi^*(q)$ denotes, respectively the number of Dirichlet characters and primitive Dirichlet characters modulo $q$. This was improved by Heath-Brown \cite{HeathBrown81} who obtained the full asymptotic expansion. The fourth moment with a power saving error-term was computed by Young \cite{YoungAnnals11}. In this paper we obtain a generalization of Heath-Brown's formula \cite[(5)]{HeathBrown81} along with a number of other asymptotic evaluations of what we call \emph{wide moments} (see Section \ref{sec:widemoment}) using connections to \emph{automorphic periods} related to the \emph{Lerch zeta function}. 

Moment calculations are important in their own right with profound connections to random matrix theory \cite{KeaSna00}, but also have applications to the important topic of non-vanishing. In our setting Chowla has conjectured that $L(\chi, 1/2)\neq 0$ for \emph{all} primitive Dirichlet characters $\chi$. The best result towards this (in the moduli aspect) is due to Mili\'{c}evi\'{c}--Khan--Ngo \cite{KhanMilNgo19} who proved that the proportion of non-vanishing of $L(\chi,1/2)$ for $\chi$ primitive modulo $p$ is at least $5/13-\eps$ (for all $\eps>0$) when $p$ is a prime tending to infinity. We obtain a different kind of non-vanishing result; so-called \emph{weak simultaneous non-vanishing} (see Theorems \ref{cor:nonvaniPY} and \ref{cor:nonvaniBettin}).

\subsection{$L$-functions and automorphic periods}
A vital fact in the study of $L$-functions is that in certain situations the central values of $L$-functions are dual (in a Fourier sense) to \emph{automorphic periods}. In some favourable situations these automorphic periods are very well-behaved, which can be used to study the $L$-functions themselves. %In particular this has been used to obtain very strong non-vanishing results (Chinta, Rohrlich). 
To illustrate this, let $f$ be a weight $2$ Hecke (normalized) newform of level $N$ (with rational Hecke eigenvalues), then it follows from the \emph{Birch--Stevens formula} \cite[Theorem 2.3]{MaRu19} that we can write for a primitive Dirichlet character $\chi \modulo q$:
$$ \tau(\overline{\chi})L(f\times \chi, 1/2)=\sum_{a\modulo q} \overline{\chi}(a) 2\pi i\int_{a/q}^{i\infty} f(z)dz,  $$
where $L(f\times \chi,s)=\sum_{n\geq 1} \lambda_f(n)\chi(n)n^{-s}$ with $\lambda_f(n)$ the Hecke eigenvalues of $f$ and $\tau(\overline{\chi})$ is a Gau{\ss} sum. The point being now that the automorphic periods $\int_{a/q}^{i\infty} f(z)dz$, known as \emph{modular symbols}, are rational multiples of a transcendental number depending only on $\chi(-1)\in \{\pm 1\}$ (more precisely the N\'{e}ron periods of the corresponding elliptic curve). This implies that non-vanishing is Galois invariant, which was utilized by Rohrlich \cite{Rohrlich89} and subsequently Chinta \cite{Chinta02} to obtain very strong non-vanishing in this set-up. In particular it is known that for Dirichlet characters $\chi$ of prime moduli, $100 \%$ of the central values $L(f\times \chi,1/2)$ are non-vanishing (in big contrast to the situation of Dirichlet $L$-functions as we saw above). In a different direction, the periods associated to $f$ are also very well-behaved from a statistical point of view; they are normally distributed when ordered by denominator as proved by Petridis--Risager \cite{PeRi}, see also \cite{Nordentoft20.4}, \cite{BeDr19} for results on higher weights and \cite{LeeSun19}, \cite{ConstNor22} for residual distribution results.  
  
In this paper we consider the case of Dirichlet $L$-functions in which case Galois invariance is not present. We do however still have the following {\lq\lq}period formula{\rq\rq} for $\chi$ primitive modulo $q$:
\begin{equation}\label{eq:periodformula} \tau(\overline{\chi})L(\chi, 1/2)=\pi^{-1/2}\sum_{a\modulo q} \overline{\chi}(a)\int_0^\infty \frac{t^{-1/2} e^{-t}}{1-e(a/q)e^{-t}}dt,   \end{equation}
where $e(x):=e^{2\pi i x}$. Again we interpret this as Dirichlet $L$-functions being dual to the {\lq\lq}periods{\rq\rq} $\int_0^\infty \frac{t^{-1/2} e^{-t}}{1-e(a/q)e^{-t}}dt$. These periods do however not have any nice rationality properties (Galois invariance \emph{is} however true for $ \tau(\overline{\chi})L(\chi, s)$ with $s$ a negative integer, which is crucial in the theory of $p$-adic $L$-functions \cite{LeopoldtKubota64}). These periods are on the other hand very well behaved as complex numbers, in the sense that the map    
$$ (0,1)\ni \alpha \mapsto \int_0^\infty \frac{t^{-1/2} e^{-t}}{1-e(\alpha)e^{-t}}dt, $$
is continuous (and in fact locally analytic). The above function, known as the \emph{periodic zeta function}, is a special case of the \emph{Lerch zeta function}. The main technical work in this paper will be to understand the analytic properties of this classical object, which we will then use to study Dirichlet $L$-functions. There is an extensive literature on the Lerch zeta function (e.g. \cite{LagLi12}, \cite{LauGar02}), but the specific properties that we will need do not seem to be available. 
%Finally we will mention that the formula (\ref{eq:periodformula}) can be seen as a {\lq\lq}global version{\rq\rq}  

\subsection{Wide moments}\label{sec:widemoments}
The present paper is a companion to \cite{Nordentoft24} in which certain \emph{wide moments} are calculated for Rankin--Selberg $L$-functions. In both cases the starting point is the following basic Fourier theoretical observation: Let $G$ be a finite abelian group and denote by $\widehat{G}$ the group of characters of $G$. Given a map $L: G\rightarrow \C$, we denote by $\widehat{L}: \widehat{G}\rightarrow \C$ the Fourier transform:
$$ \chi\mapsto \frac{1}{|G|}\sum_{g\in G} \overline{\chi}(g) L(g). $$  Then given maps $L_1,\ldots, L_m: G\rightarrow \C$, we have the following key identity:
\begin{equation}
\label{eq:fouriertrick}  \sum_{(\chi_i)_i\in \wide(\widehat{G},m;\chi)} \prod_{i=1}^m \widehat{L_i}(\chi_i)=\frac{1}{|G|}\sum_{g\in G} \overline{\chi}(g)\prod_{i=1}^m L_i(g),
\end{equation}
where
\begin{equation}\label{eq:Wide} \wide(\widehat{G},m;\chi):=\{ (\chi_1,\ldots, \chi_m)\in (\widehat{G})^m: \chi_1\cdots \chi_m=\chi \}. \end{equation}
If the maps $L_i$ are well-behaved and in particular if one can calculate the moments on the righthand side of (\ref{eq:fouriertrick}), then one gets {\lq\lq}for free{\rq\rq} a calculation of the dual side, which is what we call the \emph{$\chi$-twisted wide moment of $\widehat{L_1},\ldots, \widehat{L_m}$} (i.e. the lefthand side of (\ref{eq:fouriertrick})). 

The moments of automorphic periods which appear in the present paper and in \cite{Nordentoft24} both fit into the following framework; one has a sequence of groups which all embed into some fixed topological space $X$:
$$G_1,G_2,G_3, \ldots \hookrightarrow  X, $$
in a way such that the maps (that appear in the wide moments) $L_{i,j}:G_j\rightarrow \C$ are restrictions (to $G_j\hookrightarrow X$) of fixed continuous functions $L_i:X\rightarrow \C$ and such that $G_j$ equidistribute  as $j\rightarrow \infty$ with respect to some measure $\mu$ on $X$  (in \cite{Nordentoft24} the $L_i$ are the relevant Maa{\ss} forms evaluated at Heegner points on the modular curve with hyperbolic measure whereas in the present paper we consider Lerch zeta functions evaluated at rational numbers considered as points on the unit circle with Lebesgue measure).  
\subsection{Statement of results}\label{sec:state}
To state our result we will use the following short-hand given an integer $q\geq 2$ and a Dirichlet character $\chi$ modulo $q$:
\begin{equation} \wide(q,m;\chi):=\{ \chi_1,\ldots, \chi_m \text{ Dirichlet characters mod $q$}: \chi_1\cdots \chi_m=\chi \},   \end{equation}
and
\begin{equation} \label{eq:wide*}\wide^*(q,m;\chi):=\{ (\chi_i)_i \in \wide(q,m;\chi): \chi_i\text{ primitive} \}.   \end{equation}%\wide\left(\widehat{\left( \Z/q\Z\right)^\times},m;\chi\right)
We obtain a number of wide moment calculations averaging over $\wide(q,m;\chi)$ and related families (see Section \ref{sec:widemoment}). In particular, we obtain the following asymptotic expansion.  
%\begin{thm}For $i=1,\ldots, m$ with $m\geq 2$, let $\chi_i\modulo q_i$ be Dirichlet characters. Then we have  
%\begin{align}\label{eq:widemomentmain}
%\frac{1}{\varphi(\ell)^{m-1}}\sum_{(\psi_i)_i\in \wide(\ell,m;1) } \prod_{i=1}^m L(\chi_i\psi_i,1/2)=\sum_{d|\ell} \mu(\ell/d) T_{\overline{d},\chi_i}(d),      
%\end{align}
%where $T_{\overline{d},\chi_i}$ depends only on the residue class $\overline{d}$ modulo $q:=\mathrm{lcm}(q_1,\ldots, q_m)$ and one has given $a\modulo q$ the asymptotic expansion for any $K\geq 1$:
%$$T_{a,\chi_i}(d)= L(\chi,  \frac{m}{2})+c_{m,a} d^{1-m/2} \log d +\sum_{k=1}^{K-1} c_{k,m,a}d^{-k/2} +O_{m, K ,a}(d^{-K/2}),$$
%where $\chi:=\chi_1\cdots \chi_m$. 
%\end{thm}
\begin{thm}\label{thm:asympexpintro}For $i=1,\ldots, m$ with $m\geq 2$, let $\psi_i\modulo \ell_i$ be Dirichlet characters. Then we have  
\begin{align}\label{eq:widemomentmain}
\frac{1}{\varphi(q)^{m-1}}\sum_{(\chi_i)_i\in \wide(q,m;1) } \prod_{i=1}^m L(\psi_i\chi_i,1/2)=\sum_{d|q} \mu(q/d) T_{\overline{d},\psi_i}(d),      
\end{align}
where $T_{\overline{d},\psi_i}$ depends only on the residue class $\overline{d}$ of $d$ modulo $\ell:=\mathrm{lcm}(\ell_1,\ldots, \ell_m)$ and on the Dirichlet characters $\psi_i$. Furthermore, given $a\modulo \ell$, we have the following asymptotic expansion for any $K\geq 1$:
$$T_{a,\psi_i}(d)= L(\psi,  m/2)+c_{m,a} d^{1-m/2} \log d +\sum_{k=1}^{K-1} c_{k,m,a}d^{-k/2} +O_{m, K ,\ell}(d^{-K/2}),$$
where $\psi:=\psi_1\cdots \psi_m$ and $c_{m,a},(c_{k,m,a})_{k\geq 1}$ are certain numerical constants depending on $\psi_1,\ldots, \psi_m$ and $a\modulo \ell$.
 %, as well as
%$$T_{a,\chi_i}(d)= L(\chi,  \frac{m}{2})+O_{m}(q^md^{1-m/2}+qd^{-1/2})$$
\end{thm}
\begin{remark}
In particular for $q$ prime, the above gives an asymptotic formula in terms of $q$ where the coefficients in the expansion only depend on $q$ modulo $\ell$.
\end{remark}
\begin{remark}
In the case of trivial twists, i.e. $\psi_i=1$, the above wide moment calculation  corresponds to the moments of the Hurwitz zeta function, which were treated by Egami--Matsumoto \cite[Theorem 3]{EgamiMatsu02} using multiple zeta sums (note that they only state the result for the third moment). They do however not notice the connection to wide moments of Dirichlet $L$-functions. The case $m=2$ was done previously by Heath-Brown \cite{HeathBrown81} using an approximate functional equation. %Nobody however seems to have noticed the connection to wide moments of Dirichlet $L$-functions. %Also in the case of trivial twist (i.e. $\ell_i=1$), a similar expression for the moments of the Hurwitz zeta functions were obtained by Egami--Matsumoto \cite[Theorem 3]{EgamiMatsu02} using multiple zeta sums (note that they only state the result for the third moment). 
\end{remark}

Using our moment calculations, we obtain a number of \emph{weak simultaneous non-vanishing results}, meaning simultaneous non-vanishing for different Dirichlet $L$-functions with an {\lq\lq}algebraic condition{\rq\rq} on the Dirichlet characters (their product is fixed). See \cite[Sec. 2]{Nordentoft24} for background and a general definition. %The below is a $\GL_1\times \GL_1$ analogue of \cite[Corollary 1.4]{Nordentoft24}.
\begin{thm}\label{cor:nonvaniPY}
%Let $\ell\geq 1$ be such that all divisors $d>1$ satisfy $d\gg \ell^\eps$. 
For $i=1,\ldots, m$ with $m\geq 2$, let $\psi_i$ modulo $\ell_i$ be a primitive Dirichlet character and put $\ell^*=\max(\ell_1,\ldots, \ell_m)$. Then for any $\eps>0$ there exists a constant $C_{m,\eps}>0$ such that all $q\geq C_{m,\eps} (\ell^*)^{m/(2m-4)+\eps}$ and any character $\chi$ modulo $q$:
\begin{align}
  \label{eq:nonvaniell}&\#\{(\chi_i)_i\in \wide^*(q, m;\chi) : L(\psi_1\chi_1,1/2)\cdots L(\psi_m\chi_m,1/2)\neq 0  \}\\
  \nonumber &\gg_\eps q^{5m/6-1-\eps}(\ell^*)^{-m/6},
\end{align}
as $q\rightarrow \infty$.
\end{thm}
We note that for $m=3$, the above theorem gives 
$$ \#\{ \chi_1,\chi_2\modulo q: L(\psi_1\chi_1,1/2)L(\psi_2\chi_2,1/2)L(\psi_3\chi_1\chi_2,1/2)\neq 0  \}\gg_\eps q^{3/2-\eps}(\ell^*)^{-1/2},  $$
for $q\gg_\eps (\ell^*)^{3/2+\eps}$, with notation as above. This should be compared with the results of Zacharias \cite[p. 2]{Zach19}, who proved using the mollification method that for any (fixed) Dirichlet characters $\psi_1\mod \ell_1,\psi_2\mod \ell_2$ there exists a positive proportion of primitive Dirichlet characters $\chi\modulo q$ such that 
$$ L(\chi,1/2)L(\psi_1\chi,1/2)L(\psi_2\chi,1/2)\neq 0. $$   
We note that the mollification method does not give any information on the algebraic structure of the non-vanishing set (of Dirichlet characters). In fact as  explained in \cite[Sec. 2]{Nordentoft24}, one needs to know non-vanishing of each individual $L(\psi_i\chi,1/2)$ (with $i=1,\ldots,m$) for at least $50\%$ of all $\chi \modulo q$ in order to conclude (by purely combinatorial means)  that the set in (\ref{eq:nonvaniell}) is \emph{non-empty}, even. As mentioned above $50\%$ non-vanishing is not known in the case of Dirichlet $L$-functions (in the $q$-aspect).   

If we restrict to trivial $\psi_i$ and $\chi$ above, we can prove a better lower bound for the non-vanishing set due to a second moment calculation of Bettin \cite{Be17}. Furthermore, we can even make the estimate uniform in the width of the moment $m$. 
\begin{thm}\label{cor:nonvaniBettin}
Let $q $ be prime and $m\geq 3$. Then we have uniformly in the range $m=o(q^{1/2}\log q)$ that
\begin{align*}
  \label{eq:nonvaniBettin}\#\{(\chi_i)_i\in \wide^*(q, m;1) : L(\chi_1,1/2)\cdots L(\chi_m,1/2)\neq 0  \} \gg q^{m-1}/(\log q)^m, 
\end{align*}
as $q\rightarrow \infty$.
%where the implied constant depends only on $$. %The same is true if we require that all $\chi_i$ are non-principal.
\end{thm}
\begin{remark}
The restriction to prime moduli in Theorem \ref{cor:nonvaniPY} is not essential, but we have restricted to this case for simplicity of exposition.
%Note that the total number of tuples $(\chi_i)\in \wide(q,m;1)$ containing \emph{at least} one non-primitive character is $O(mq^{m-2})$ and so we may restrict to primitive characters in Corollary \ref{cor:nonvaniBettin}. 
\end{remark}
\subsection*{Acknowledgements}
The author's research was supported by the German Research Foundation under Germany's Excellence Strategy EXC-$2047/1$-$390685813$, as well as the Independent Research Fund Denmark DFF-1025-00020B.
\section{Analytic properties of Lerch zeta functions}
The results of this paper follows by a detailed study of the \emph{Lerch zeta function}, which for $\alpha\in \R$, $c>0$ and $\Re s>1$ is defined as the following Dirichlet series: 
\begin{equation}\label{eq:defLerch} \zeta(\alpha, c, s)= \sum_{n\geq 0}\frac{e(n\alpha)}{(n+c)^s}, \quad e(x)=e^{2\pi i x}. \end{equation}
The Lerch zeta function admits meromorphic continuation to the entire complex plane and satisfies the following functional equation \cite[Theorem 3.2]{LauGar02}
\begin{equation}\label{eq:FE}	
 \zeta(\alpha, c, 1-s)= \frac{\Gamma(s)}{(2\pi)^s}\left( e\left( s/4-\alpha c\right)\zeta(-c, \alpha, s)+ e\left(- s/4+(1-\alpha) c\right)\zeta(c, 1-\alpha, s)  \right),
\end{equation}
valid for $0<\alpha<1$ and $c>0$. The analytic continuation to $\Re s>0$ for $\alpha\notin \Z$ is given by the following elementary {\lq\lq}period integral{\rq\rq}:
\begin{equation}\label{eq:intrep} \zeta(\alpha, c, s)=\Gamma(s)^{-1}\int_0^\infty t^{s-1} \frac{e^{-ct}}{1-e(\alpha)e^{-t}}dt,\end{equation}
as can be seen easily by writing $(1-e(\alpha)e^{-t})^{-1}=\sum_{n\geq 0}e(n\alpha)e^{-tn}$. Another expression for the analytic continuation can be obtained by using a Taylor expansion (as we will see below). Both expressions will be useful in order to extract analytic properties. %We will show another (of the many) way(s) to obtain analytic continuation, which will be very useful for extracting the analytic properties that we will need.
%In particular we see that $L(\alpha, c, 1/2)$ does not converge to $L(0,1,1/2)=\zeta(1/2)$ as $(\alpha, c)\rightarrow (0,1)$. We thus needs a different treatment in that case that $\alpha=0$, which corresponds to the Hurwitz zeta function.
\begin{prop}\label{prop:taylor} For $\alpha\in \Z$, $0<c<1$ and $\Re s>0$, we have the following absolutely and locally uniformly convergent series representation (with a pole at $s=1$ for $\alpha \in \Z$):  
\begin{equation}\label{eq:expinc} \zeta(\alpha, c, s)=c^{-s}+e(\alpha)\sum_{n\geq 0} (-c)^n \binom{s+n-1}{n} \zeta(\alpha,1, s+n). \end{equation}
In particular, we have for $0<\alpha,c<1$:
\begin{equation}\label{eq:expinc2}  \zeta(\alpha, c, 1/2)= \frac{1+i}{2} \alpha^{ -1/2}+ \frac{1-i}{2}(1-\alpha)^{-1/2}+c^{- 1/2} +\sum_{n\geq 0}  C_n(\alpha)c^n,\end{equation}
where the coefficients satisfy the bound $C_n(\alpha)\ll n^{-1/2}$ uniformly in $\alpha$.

%In particular, we have the following asymptotic expansion for $0<\alpha<1$ in terms of $0<c<1$:
%\begin{equation}\label{eq:expinc3}  \zeta(\alpha, c, 1/2)=\frac{1+i}{2}\alpha^{-1/2}+c^{-1/2}+\sum_{n\geq 0}  C_n(\alpha)c^n,\end{equation}
%where the coefficients satisfy the bound $C_n(\alpha)\ll n^{-1/2}$ uniformly in $\alpha$.% and $\eps \leq \Re s< 1$.% and $|\Im s|<\eps^{-1}$. 
\end{prop}
\begin{proof}
We have by termwise differentiation for $\Re s>1$:
$$  \frac{\partial}{\partial c}\zeta(\alpha, c, s)= -s\zeta(\alpha, c, s+1). $$
By writing $\zeta(\alpha, c, s)=c^{-s}+e(\alpha)\zeta(\alpha, c+1, s)$, we conclude (\ref{eq:expinc}) for $\Re s>1 $ by Taylor expanding $\zeta(\alpha, c+1, s)$ around $c=0$. Now using the trivial bound $ \binom{s+n-1}{n}\leq |s|^n$ together with $ \zeta(\alpha,1, s+n)\ll_\eps 1$ for $\Re s+n>1+\eps$, we see that (\ref{eq:expinc}) converges absolutely and locally uniformly for $\Re s>0$ to a meromorphic function in $s$ with poles exactly where $\zeta(\alpha,1, s)$ has poles. For $\alpha\notin \Z$, it follows from (\ref{eq:intrep}) that    
$\zeta(\alpha,1, s)$ is analytic in the half-plane $\Re s>0$ and for $\alpha\in \Z$ there is a unique simple pole at $s=1$ as $\zeta(0,1,s)=\zeta(s)$ is the Riemann zeta function.

In particular it follows from (\ref{eq:expinc}) with $\alpha=0$ and $s=1/2$ that we have 
$$\zeta(0, c, 1/2)=c^{-1/2}+g_1(c),$$ 
where $g_1:[0,\infty)\rightarrow \C$ is continuous and locally analytic (also at $c=0$). Using the functional equation (\ref{eq:FE}) this implies that for $\alpha\in (0,1)$:
$$\zeta(\alpha, 1, 1/2)=\frac{\Gamma(1/2)}{(2\pi)^{1/2}} \left(e(1/8) \alpha^{-1/2}+ e(-1/8) (1-\alpha)^{-1/2}\right)+g_2(\alpha),$$ where $g_2(\cdot):[0,1]\rightarrow \C$ is a continuous functions which is locally analytic.

Plugging this back into (\ref{eq:expinc}) (now with a general $\alpha$) we conclude (\ref{eq:expinc2}). The bound for the coefficients $C_n(\alpha)$ follows from Stirling's formula. 
%In particular we see that 
%$$     \zeta(0, c, 1/2)= c^{-1/2}+O(1),   $$
%as $c\rightarrow 0$. Using the functional equation with $s=1/2$, this implies 
%$$ \zeta(\alpha, 1, 1/2)= \frac{1+i}{2} \alpha^{-1/2}+O(1),$$
%for $0<\alpha<1/2$ and $\alpha \rightarrow 0$. 
%Now (\ref{eq:expinc2}) follows since $ \zeta(\alpha,1, 1/2+n)\ll 1$ uniformly in $\alpha$ and $n\geq 1$ (by absolute convergence) and $ \binom{n-1/2}{n}\ll n^{-1/2}$.   
\end{proof}

%Using this we can get asymptotic expansions for all wide moments (for the second moment) this recovers Heath-Browns calculation (which was done slightly differently).
\subsection{Lerch zeta functions as Fourier transforms of Dirichlet $L$-functions}
The connection between Lerch zeta functions and Dirichlet $L$-functions is a very classical topic (see e.g. \cite[Chapter 9]{Davenport00}). Below we state a general form of this connection, which can be seen as a $\GL_1$-version of the so-called Birch--Stevens formula (see \cite[Proposition 6.1]{Nordentoft20.4}).  
\begin{prop}\label{prop:BirchStevens}
Let $\chi$ and $\psi$ be Dirichlet characters modulo $q$ and $\ell $, respectively. Then we have
\begin{align}
\sum_{a\modulo q}\, \sum_{1\leq b \leq \ell } \chi(a)\psi(b) e(ab/q)\zeta\left( \frac{a\ell }{q}, \frac{b}{\ell },s\right)= \ell ^{s}\tau(\chi^*)\nu_s(\chi^*, \psi, q/q^*) L(\overline{\chi^*} \psi, s),
\end{align}
for all $s\in \C$, where $\chi^*$ modulo $q^*$ is the unique primitive Dirichlet character inducing $\chi$. The weight $\nu_s$ is given by the following convolution product:
\begin{equation}  \nu_s(\chi^*, \psi,  q/q^*) := \left[((\cdot )^{1-s} \psi)\ast (\chi^* \mu)\right](q/q^*)=\sum_{d| q/q^*} d^{1-s} \psi(d) \mu(q/(q^* d))\chi^*(q/(q^* d)). \end{equation}
 This formula also holds if either $q=1$ or $\ell =1$.
\end{prop}
\begin{proof}
We note that for $\Re s>1$ and $1\leq b \leq \ell $:
$$  \ell^{-s}e(ab/q )\zeta\left( \frac{a\ell}{q}, \frac{b}{\ell },s\right) = \sum_{\substack{n\geq 1\\ n\equiv b\modulo \ell}} e\left(\frac{na}{q}\right)n^{-s}, $$
and thus the identity follows when $\Re s>1$ by interchanging the sums and using basic facts about Gau{\ss} sums as in the proof of \cite[Proposition 6.1]{Nordentoft20.4}. Now the result follows for $s\in \C$ by uniqueness of analytic continuation.  
\end{proof}

\subsection{Twisted version of the Lerch zeta function}
Since we have the two directions $(\alpha, c)$ to vary, we can use Proposition \ref{prop:BirchStevens} to derive the analytic properties of \emph{twisted} versions of Lerch zeta functions. 

Let $\chi$ be a primitive Dirichlet character modulo $q$ and let $b,\ell$ be integers with $(b,\ell)=1$. Then we define the associated \emph{twisted Hurwitz zeta function} as: 
\begin{align} L(\chi, b/ \ell, s):=  \ell^s \sum_{\substack{n\geq 1,\\ n\equiv b\modulo \ell}} \frac{\chi(n)}{n^s}, \end{align}
which converges absolutely for $\Re s>1$. By Proposition \ref{prop:BirchStevens} it is easy to see that  
\begin{equation}\label{eq:periodformula1} L(\chi, b/\ell, s)= \tau(\overline{\chi})^{-1}\sum_{a\modulo q} \overline{\chi}(a) e(ab/q) \zeta(a\ell/q,b/\ell,s),  \end{equation}
which shows that $L(\chi, b/\ell, s)$ admits meromorphic continuation to the entire complex plane. 

Similarly given a Dirichlet character $\psi$ modulo $q$, we define the \emph{twisted periodic zeta function}:
\begin{equation}
L(\alpha, \psi, s):=   \sum_{n\geq 1} \frac{e(n\alpha)\psi(n)}{n^s},\quad \Re s>1
\end{equation}
As above we have meromorphic continuation given by:
\begin{equation}\label{eq:MCtwistedperiodic}
L(\alpha, \psi, s)=\ell^{-s}\sum_{b=1}^\ell \psi(b) e(b\alpha) \zeta(\ell\alpha,b/\ell,s)
\end{equation}
Also we see that $(0,1)\ni \alpha\mapsto L(\alpha, \psi, s)$ defines a continuous function using the above observations about the Lerch zeta function. 

The two different twisted versions of the Lerch zeta function have quite different analytic properties as can be seen from the following propositions. 
%In the first case we will instead of relying on the Taylor expansion as above, we will employ the 

%\begin{proof}Alternative proof:
%Plugging (\ref{eq:expinc2}) into (\ref{eq:periodformula1}), we see that as $\alpha\rightarrow 0^+$, we can write 
%$$L(\alpha, \psi, s)= \ell^{-s} \sum_{b=1}^\ell \overline{\psi}(b) e(b\alpha) \left( \frac{\Gamma(1-s)}{(2\pi)^{1-s}} e((1-s)/4)   \{ \ell\alpha \}^{s-1}+g(\alpha, b/\ell,s ) \right),$$
%where $\{x\}$ denotes the fractional part of $x$ and $g(\cdot, b/\ell,s):[-1/2,1/2]\rightarrow \C$ is continuous and locally analytic. Furthermore for $\alpha$ (and $b/\ell$) fixed $g(\alpha, b/\ell,s)$ is analytic in $s$. We observe that as
%$$\sum_{b=1}^\ell \overline{\psi}(b) e(b\alpha)=\sum_{b=1}^\ell \overline{\psi}(b) (1+O(b\alpha))=O_\ell(\alpha),  $$
%the righthand side
%\end{proof}

 %Furthermore using the above one derives easily the following asymptotic behaviour.%This also allows one to define the LHS for any real $0<b<l$ (in analogy with Bump's calculation for Dirichlet characters in the beginning of his book.)
\begin{prop}\label{prop:expansiontwisted}
Let $\chi$ be a primitive Dirichlet character modulo $q$, then we have for $(\ell,q)=1$ and $1\leq b<\ell$:
\begin{equation}\label{eq:expansiontwisted}  L(\chi, b/\ell, 1/2)= \chi(b)(b/\ell)^{-1/2}+\sum_{n\geq 0} C_{n,\chi}(b, \ell \modulo q)(b/\ell)^n,   \end{equation}
where $C_{n,\chi}(b, \ell \modulo q)\ll (q/n)^{1/2}$ are certain coefficients depending only on $b, \ell$ modulo $q$.

Furthermore, we have for any $\eps>0$
\begin{equation}
\label{eq:asympL} L(\chi, b/\ell, 1/2)= \chi(b)(b/\ell)^{-1/2}+O_\eps(q^{1/4+\eps}),
\end{equation}
uniformly for $1\leq b\leq \ell$.
\end{prop}
\begin{proof}
Combining (\ref{eq:periodformula1}) and (\ref{eq:expinc2}), we get
\begin{align*}&L(\chi, b/\ell, 1/2)\\
&= \tau(\overline{\chi})^{-1}\sum_{a\modulo q} \overline{\chi}(a) e(ab/q) \Biggr( (b/\ell)^{-1/2}+\frac{1+i}{2}\{ a\ell/q\}^{-1/2}+\frac{1-i}{2}(1-\{ a\ell/q\})^{-1/2}\\
&\qquad \qquad\qquad \qquad\qquad \qquad +\sum_{n\geq 0}  C_n(\{ a\ell/q\}) (b/\ell)^n\Biggr),\end{align*}
where $\{x\}$ denotes the fractional part of $x$. We notice that the coefficients above only depend on $b,\ell$ modulo $q$. We get the wanted main term using standard properties of Gau{\ss} sums. Finally the bound on the coefficients follows directly by using the bound $C_n(\{ a\ell/q\},1/2)\ll n^{-1/2}$ and $|\tau(\overline{\chi})|=q^{1/2}$, which proves (\ref{eq:expansiontwisted}). 

By (\ref{eq:expinc2}), we can write
$$ \zeta(\alpha, c,1/2)= c^{-1/2}+f(\alpha,c), $$
where $f$ satisfies 
$$f(\alpha,c)\ll \alpha^{-1/2}+(1-\alpha)^{-1/2},\quad \partial_\alpha f(\alpha,c)\ll \alpha^{-3/2}+(1-\alpha)^{-3/2}. $$ 
Inserting the above into the formula (\ref{eq:periodformula1}), we get:
\begin{align}
\label{eq:number1}L(\chi, b/\ell,1/2)&= \tau(\overline{\chi})^{-1} \left(\sum_{a\modulo q} \overline{\chi}(a) e(ab/q)\left((b/\ell)^{-1/2}+f(a\ell/q,b/\ell) \right)  \right)\\
\nonumber &=\chi(b) (b/\ell)^{-1/2}+ \frac{\chi(\ell)}{\tau(\overline{\chi})}\sum_{a\modulo q} \overline{\chi}(a) e(ab\overline{\ell}/q)f(a/q,b/\ell),\end{align}  
where we used a change of variable $a\leftrightarrow a\overline{\ell}$. By partial summation we get 
\begin{align}
\label{eq:floorsum}&\sum_{a=1}^{\lfloor q/2 \rfloor} \overline{\chi}(a) e(ab\overline{\ell}/q)f(a/q,b/\ell)\\
\nonumber &= S(\lfloor q/2 \rfloor, b\overline{\ell})f(\lfloor q/2 \rfloor/q,b/\ell) - \frac{1}{q}\int_1^{\lfloor q/2 \rfloor} S(x, b\overline{\ell})\partial_x f(x/q,b/\ell) dx,\end{align}
where $S(x,r):=\sum_{1\leq a \leq x} \overline{\chi}(a) e(ar/q)$. By P\'{o}lya--Vinogradov (or more precisely a straight forward generalization) we have the estimate
$$S(x,r)\ll_\eps \min (x, q^{1/2+\eps}),\text{ uniformly for $x\geq 1$ and $r\in \Z$},  $$
and thus (\ref{eq:floorsum}) is bounded by:
$$ \ll_\eps q^{1/2}+\frac{1}{q}\left(\int_1^{q^{1/2}} x (x/q)^{-3/2} dx +q^{1/2+\eps}\int_{q^{1/2}}^{\lfloor q/2 \rfloor}  (x/q)^{-3/2} dx \right)\ll q^{3/4+\eps}.$$
The contribution from $\lfloor q/2 \rfloor<a\leq q-1$ can be bounded similarly. Plugging these estimates into (\ref{eq:number1}) gives the wanted using $|\tau(\overline{\chi})|=q^{1/2}$.
\end{proof}
\begin{prop}\label{prop:analyticprop}
Let $\psi$ be a primitive Dirichlet character modulo $\ell$. Then 
$$[0,1) \ni \alpha\mapsto L(\alpha, \psi, 1/2),$$ 
defines a locally analytic function for $\alpha \notin \{\tfrac{b}{\ell}: 1\leq b\leq \ell-1\}$, satisfying for $0\leq \alpha <1$:  
\begin{equation}\label{eq:twistperiodic2} 
L(\alpha, \psi, 1/2)= \frac{\tau(\psi)\overline{\psi}(\lfloor \ell \alpha\rfloor)(1+i)}{2\ell^{1/2}} \{\ell \alpha\}^{-1/2}+\frac{\tau(\psi)\overline{\psi}(\lceil \ell \alpha\rceil)(1-i)}{2\ell^{1/2}} (1-\{\ell \alpha\})^{-1/2}+O_\ell(1),\end{equation}
where $\{x\}=x-\lfloor x\rfloor$ denotes the fractional part of $x$  and $\tau(\psi):=\sum_{a\in (\Z/\ell \Z)^\times}\psi(a)e(a/\ell)$ is a Gau{\ss} sum.

Furthermore, as $\alpha\rightarrow 0$ we have
\begin{equation}\label{eq:twistperiodic2} 
L(\alpha, \psi, 1/2)=L(\psi,1/2)+O_{\ell,\eps}(|\alpha|^{1/2-\eps}).
\end{equation}
\end{prop}	
\begin{proof} It follows directly from the period formula (\ref{eq:intrep}) for the Lerch zeta function that for $c>0$ the Lerch zeta $\alpha\mapsto \zeta(\alpha,c,1/2) $ defines a periodic and locally analytic function on $\R\setminus \Z$. Thus we conclude from the formula (\ref{eq:MCtwistedperiodic}) that $\alpha\mapsto L(\alpha,\psi,s)$ is locally analytic for $\alpha \not\equiv \tfrac{b}{\ell}\modulo 1$ with $0\leq b< \ell$. 

To understand the behavior around $\alpha=b/\ell$ we observe that by equation (\ref{eq:expinc2}) we have for $0<\alpha<1$:
$$\zeta(\alpha, \tfrac{b}{\ell},1/2)= \frac{1+i}{2}\alpha^{-1/2}+\frac{1-i}{2}(1-\alpha)^{-1/2}+O_\ell(1).$$
By Taylor expansion we have:
$$ e(b\alpha)=e(\tfrac{b}{\ell} \ell\alpha)=e(\tfrac{b}{\ell} (\lfloor \ell\alpha\rfloor+\{\ell\alpha\}))=e(\tfrac{b}{\ell} \lfloor \ell\alpha\rfloor)+O_\ell(\{\ell\alpha\}),  $$
and similarly
$$ e(b\alpha)=e(\tfrac{b}{\ell} \lceil \ell\alpha\rceil)+O_\ell(1-\{\ell\alpha\}). $$
Plugging this into equation (\ref{eq:MCtwistedperiodic}) we arrive at
\begin{align*}
\ell^{1/2}L(\alpha, \psi, 1/2)=& \sum_{b=1}^\ell \psi(b)e(b\alpha) \left(\frac{1+i}{2}\{\ell \alpha\}^{-1/2}+\frac{1-i}{2}(1-\{\ell \alpha\})^{-1/2}+O_\ell(1)\right)\\
= &\frac{1+i}{2}\{\ell\alpha\}^{-1/2} \sum_{b=1}^\ell \psi(b)(e(\tfrac{b}{\ell}\lfloor \ell \alpha\rfloor )+O_\ell(\{\ell\alpha\}))\\
&+\frac{1-i}{2}(1-\{\ell\alpha\})^{-1/2}\sum_{b=1}^\ell \psi(b)(e(\tfrac{b}{\ell}\lceil \ell \alpha\rceil)+O_\ell(1-\{\ell\alpha\}))+O_\ell(1)\\
=&\frac{1+i}{2}\{\ell\alpha\}^{-1/2} \sum_{b=1}^\ell \psi(b)e(\tfrac{b}{\ell}\lfloor \ell \alpha\rfloor )+\frac{1-i}{2}(1-\{\ell\alpha\})^{-1/2}\sum_{b=1}^\ell \psi(b)e(\tfrac{b}{\ell}\lceil \ell \alpha\rceil)+O_\ell(1).
\end{align*}
We get the wanted conclusion by the standard properties of Gau{\ss} sums 
$$\sum_{b=1}^\ell \psi(b)e(bc/\ell)=\psi(\overline{c}) \tau(\psi)=\overline{\psi}(c) \tau(\psi), $$
using here that $\psi$ is primitive.

Note that we get immediately from the above that $L(\alpha,\psi,1/2)$ is bounded as $\alpha\rightarrow 0$. We will obtain the finer estimate (\ref{eq:twistperiodic2}) using the integral representation (\ref{eq:intrep}). First of all for $\alpha\in [-1/2\ell,1/2\ell]$ and $t\in (0,1)$, we have: 
\begin{equation}\label{eq:bndfordenom}|1-e(\ell \alpha)e^{-t}|\geq |1-e^{-t}+ie^{-t}\sin 2\pi \ell \alpha |\gg \max (t,|\ell \alpha|)  .\end{equation}
To understand the behavior at $\alpha=0$, we use again the integral representation and write:
\begin{align*}
 L(\alpha, \psi, 1/2)-L( \psi, 1/2)=\frac{1}{(\pi\ell)^{1/2}}(C_1+C_2+C_3),
\end{align*}
where 
\begin{align*}
&C_1=\int_0^1 \left( \sum_{b=1}^\ell \psi(b) (e(b\alpha)-1) e^{-\frac{b}{\ell} t} \right)t^{-1/2}\frac{1}{1-e(\ell\alpha)e^{-t}} dt\\%\ll_\ell \int_0^1\frac{|\alpha|t^{-1/2}}{\max(t, |\alpha|)}dt\ll_\eps |\alpha|^{1/2-\eps},\\
&C_2=\int_0^1 \left( \sum_{b=1}^\ell \psi(b)  e^{-\frac{b}{\ell} t} \right)t^{-1/2}\frac{e^{-t}(e(\ell\alpha)-1)}{(1-e^{-t})(1-e(\ell\alpha)e^{-t})} dt\\
&C_3= \int_1^\infty \left( \frac{\sum_{b=1}^\ell \psi(b) e(b\alpha) e^{-\frac{b}{\ell} t}}{1-e(\ell\alpha)e^{-t}}-\frac{\sum_{b=1}^\ell \psi(b) e^{-\frac{b}{\ell} t}}{1-e^{-t}}\right) t^{-1/2}dt.
\end{align*}
Using $e^{-\frac{b}{\ell} t}=1+O_\ell(t)$, $e(b\alpha)= 1+O_\ell(|\alpha|)$, $\sum_{b=1}^\ell \psi(b)=0$ and (\ref{eq:bndfordenom}), it follows that $C_i\ll_{\ell,\eps} |\alpha|^{1/2-\eps} $ which gives the wanted. 
\end{proof}

\section{Moments of Lerch zeta functions}\label{sec:momentsLerch}
Using the above derived properties of the Lerch zeta function, we can calculate a number of different moments, which will have interesting consequences for Dirichlet $L$-functions. We state them here for the central values $s=1/2$, but the methods are (clearly) flexible and can deal with any $0<\Re s<1$ with $|\Im s|\ll 1$. We can vary $L(\alpha, c, 1/2)$ both in the $\alpha$ and the $c$ direction, as well as both directions simultaneously.   

\subsection{Moments of the Hurwitz zeta function} First of all we consider the case of varying the $c$-direction. In this case $\zeta(0, c,s)$ is known as the \emph{Hurwitz zeta function} and the moments of this function has been studied intensively previously. In particular, Egami--Matsumoto \cite{EgamiMatsu02} obtained an asymptotic expansion in $\ell$ for the moments $\sum_{a=1}^{\ell-1} |\zeta(0, a/\ell,1/2)|^{2n}$ building on multiple zeta function-techniques of Zagier \cite{ZagierECM94}. When $n=1$ this was previously achieved by Heath-Brown \cite{HeathBrown81} using an approximate functional equation approach. We introduce a new approach relying on the Euler--Maclaurin formula \cite[Theorem 4.2]{IwKo}:
\begin{align} 
\nonumber \sum_{n=L+1}^M f(n)= \int_L^M f(x)dx+\sum_{k=1}^K \frac{B_k}{k!}(f^{(k-1)}(M)-f^{(k-1)}(L))+R_K, \end{align}
where $K$ and $L<M$ are integers and $f:[L,M]\rightarrow \C$ is a $K$-times differential function and the {\lq\lq}error-term{\rq\rq} is given by:
\begin{align} \label{eq:EMerror}R_K=(-1)^{K+1}\int_L^M \frac{\psi_K(x)}{K!}f^{(K)}(x)dx\ll_K \int_L^M |f^{(K)}(x)|dx. \end{align}
Here $B_k$ denotes the $k$th Bernoulli number and $\psi_K$ the $K$th Bernoulli polynomial. 
 
Our method allows us to generalize the previous results to \emph{twisted} Hurwitz zeta function and furthermore provides a simplified proof of the results of Egami--Matsumoto and Heath-Brown. % and furthermore generalizes to the moments of \emph{twisted} Hurwitz zeta function.    
\begin{thm}\label{thm:asympexp}
For $i=1,\ldots,m$ with $m\geq 2$, let $\chi_i\modulo q_i$ be a Dirichlet character and put $q:= \mathrm{lcm}(q_1,\ldots, q_m)$. Then we have for $\ell\equiv a\modulo q$ and any $K\geq 1$, the following asymptotic expansion:
\begin{align}
\nonumber&\sum_{b=1}^{\ell-1} \prod_{i=1}^m L(\chi_i, b/\ell, 1/2)\\
\label{eq:asympexptwist1}&=   L\left(\chi,  \frac{m}{2}\right) \ell^{m/2}+ c_{m,a}\ell\log \ell+ \sum_{k=1}^{K-1} \ell^{m/2-k/2} c_{k,m,a}+O_{m, K ,q}(\ell^{m/2-K/2}), %+c_{a,m}\, (\log \ell)\ell
\end{align}  
where $\chi=\chi_1\cdots \chi_m$ and $c_{k,m,a}, c_{m,a}$ are some numerical constants depending only on the residue class of $\ell \modulo q$ and on $\chi_1,\ldots, \chi_m$.   
\end{thm}
\begin{proof}
We split the sum according to residue classes modulo $q$:
\begin{align}
\sum_{b=1}^{\ell} \prod_{i=1}^m L(\chi_i, b/\ell, 1/2)= \sum_{a=1}^q\sum_{b=0}^{\ell_a} \prod_{i=1}^m L\left(\chi_i, \frac{bq+a}{\ell}, 1/2\right), 
\end{align}
where $\ell_a:=\lfloor \frac{\ell-a-1}{q}\rfloor$. By Proposition \ref{prop:expansiontwisted}, we can write 
$$\prod_{i=1}^m L\left(\chi_i, \frac{bq+a}{\ell}, 1/2\right)=f_a\left(\frac{bq+a}{\ell}\right),$$
where $f_a: (0,\infty)\rightarrow \C$ is locally analytic and given by a Laurent series in $x^{1/2}$ around $x=0$ (suppressing $\chi_1,\ldots, \chi_m$ and $\ell \modulo q$ in the notation). Now given $K\geq 1$, we can write
\begin{equation}\label{eq:twopieces} f_a(x)= \sum_{k=-m}^{2K-1} c_{a,k}x^{k/2}+ g_{a,K}(x), \end{equation}
for certain constants $c_{a,k}$ with $c_{a,-m}= \prod_{i=1}^m \chi_i(a)=\chi(a)$, and where $g_{a,K}: (0,\infty)\rightarrow \C$ is $K$-times differentiable at $x=0$ (with derivatives given by a power series in $x^{1/2}$ around $x=0$) and locally analytic for $x\in (0,\infty)$. We will now apply the Euler--Maclaurin formula in two different ways to the two pieces on the righthand side of (\ref{eq:twopieces}) and show that both have the desired shape (in view of (\ref{eq:asympexptwist1})). 

First of all we get by the Euler--Maclaurin formula applied to $g_{a,K}$ and bounding the error-term using (\ref{eq:EMerror}):
\begin{align}
\label{eq:g_aK}&\sum_{b=0}^{\ell_a} g_{a,K}\left(\frac{bq+a}{\ell}\right)= g_{a,K}\left(\frac{a}{\ell}\right)+\left(g_{a,K}^{(-1)}\left(\frac{\ell_a q+a}{\ell}\right)-g_{a,K}^{(-1)}\left(\frac{a}{\ell}\right)\right)\frac{\ell}{q}\\%\int_0^{\ell_a} g_{a,K}(\frac{xq+a}{\ell}) dx\\
\nonumber&+\sum_{k=0}^K \frac{B_{k}}{k!}\left(\frac{q}{\ell}\right)^k\left(g_{a,K}^{(k)}\left(\frac{\ell_aq+a}{\ell}\right)-g_{a,K}^{(k)}\left(\frac{a}{\ell}\right)\right)+O_{K,q}\left(\ell^{1-K}\right),
\end{align}
where $g_{a,K}^{(-1)}$ denotes the antiderivative of $g_{a,K}$. Observe that 
$$ \frac{\ell_a q+a}{\ell}=1-\frac{c(\ell, a \modulo q)}{\ell},$$
where $c(\ell, a \modulo q)$ is some constant only depending on $a$ and $\ell $ modulo $q$. Thus since  $ g^{(k)}_{a,K}$ is analytic for $k=-1,0,\ldots, K$  around $x=1$ and analytic in $x^{1/2}$ around $x=0$, one easily sees that (\ref{eq:g_aK}) is indeed a power series in $\ell^{1/2}$ with an error-term of the desired shape (for $\ell$ sufficiently large in terms of $q$). Furthermore we observe that (\ref{eq:g_aK}) is $O_{K,q}(\ell)$.

Now for the second part of the sum, we use the Euler--Maclaurin formula to write for $\Re s> 1$ and any $L,K'\geq 0$:
\begin{align*}
\zeta(0,c,s)&= \sum_{n=0}^{L} (n+c)^{-s} +\sum_{n=L+1}^\infty (n+c)^{-s}\\
&= \sum_{n=0}^L (n+c)^{-s}+\frac{(L+c)^{1-s}}{1-s}+\sum_{k=1}^{K'} (L+c)^{1-s-k} \frac{B_k}{k!} \prod_{j=0}^{k-2}(s+j)\\
&\qquad \qquad\qquad  - \frac{\prod_{j=0}^{K'-1}(s+j)}{(K')!}\int_L^\infty \frac{\psi_{K'}(x)}{(x+c)^{s+K'}}dx,    
\end{align*} 
where $\psi_{K'}$ denotes the $K'$-th Bernoulli polynomial. Now we observe that the above expression provides meromorphic continuation for $\zeta(0,c,s)$ to the region $\Re s+K'>1+\eps$ (with a pole at $s=1$). In particular, if we put $c= a/q$, $K'=K+1$ and $L=\ell_a$ (as above), then we get for $-\sigma\in \{k/2: k=-m,\ldots, 2K-1\}\backslash\{-1\}$
\begin{align*}
&\sum_{b=0}^{\ell_a} \left(\frac{bq+a}{\ell} \right)^{-\sigma}= \left(\frac{q}{\ell} \right)^{-\sigma}\sum_{b=0}^{\ell_a} \left(b+a/q \right)^{-\sigma}\\
&= \left(\frac{q}{\ell} \right)^{-\sigma} \left( \zeta(0,a/q,\sigma)- \Biggr(\frac{\ell}{q}\right)^{1-\sigma}\frac{\left(1-\frac{c(\ell, a \modulo q)}{\ell}\right)^{1-\sigma}}{1-\sigma} \\
&+\sum_{k=1}^{K+1} \left(\frac{\ell}{q}\right)^{1-\sigma-k}\left(1-\frac{c(\ell, a \modulo q)}{\ell}\right)^{1-\sigma-k} \frac{B_k}{k!} \prod_{j=0}^{k-2}(\sigma+j)\Biggr) +O_{K,q}\left(  \ell^{-K} \right),
\end{align*}
by bounded the error-term as in (\ref{eq:EMerror}). Now by Taylor expanding the power functions $x^{1-\sigma-k}$ around $x=1$, we get an expression of the desired shape. For $-\sigma=-1$, we use the same approach together with the classical fact %appears on the wikipage of "Hurwitz zeta function"
\begin{align*} \lim_{\sigma\rightarrow 1} \zeta(0,a/q,\sigma) +\frac{(\ell_a+a/q)^{1-\sigma}}{1-\sigma}&= \log(\ell_a+a/q)-\frac{\Gamma'(a/q)}{\Gamma(a/q)}\\
&= \log\left(1-\frac{c(a,\ell \modulo q)}{\ell}\right)+\log \ell-\log q-\frac{\Gamma'(a/q)}{\Gamma(a/q)}. \end{align*}
When summed over $a\modulo q$ this is again of the desired shape upon Taylor expanding $\log(1-x)$ (around $x=0$).%Now we notice that the $\log \ell$ term vanishes since $\chi$ is assumed to be non-principal. 

Finally to calculate the main term (i.e. leading coefficient), we observe that for each $a\modulo q$, we get a term of magnitude $\ell^{m/2}$, exactly for $-\sigma=-m/2$ and thus the total main term is 
$$ \ell^{m/2} q^{-m/2}\sum_{a=1}^q \chi(a)\zeta(0,a/q, m/2)= \ell^{m/2} L(\chi,m/2),$$
as wanted.    
\end{proof}
In what follows we will be a bit less ambitious, in that we will not obtain the entire  asymptotic expansion, but only the main term. In return we can deal with much more general moments. Instead of using the full Euler--Maclaurin formula we will contain ourselves to the following.
\begin{lemma}\label{lem:key} We have the following asymptotic evaluations:
\begin{align}
\sum_{a=1}^{q-1}a^{-\sigma}=
\begin{cases}
 \zeta(\sigma)+O(\frac{q^{1-\sigma}}{1-\sigma}),& \sigma>1,\\
 \log q+O(1), & \sigma=1,\\
 \frac{q^{1-\sigma}}{1-\sigma}+O_\sigma(q^{-\sigma}+1),& \sigma<1.
\end{cases}
\end{align}
\end{lemma}
\begin{proof}
This is standard by the Euler--Maclaurin formula \cite[Theorem 4.2]{IwKo}.
\end{proof}
First of all we will calculate the same wide moment as in Theorem \ref{thm:asympexp} but with a general twist $\psi\modulo \ell$.
\begin{thm} \label{thm:generaltwist} Let $m\geq 3$ and let notation be as in Theorem \ref{thm:asympexp}. Then we have for any Dirichlet character $\psi\modulo \ell$ and any $\eps>0$:
\begin{align*}\sum_{b=1}^{\ell-1} \psi(b)\prod_{i=1}^m L\left(\chi_i, b/\ell, \frac{1}{2}\right) 
= \ell^{m/2} L\left(\chi\psi,  \frac{m}{2}\right)+O_{m,\eps}((q^*)^{m/4+\eps}\ell+(q^*)^{1/4+\eps}\ell^{m/2-1/2+\eps}), \end{align*}
where $q^*=\max(q_1,\ldots, q_m)$.
\end{thm}
\begin{proof}
We have by (\ref{eq:asympL}):
\begin{align*}
&\sum_{b=1}^{\ell-1} \psi(b)\prod_{i=1}^m L(\chi_i, b/\ell, 1/2)\\
&= \sum_{b=1}^{\ell-1} \psi(b)\chi(b) (b/\ell)^{-m/2}+O_m\left(\sum_{i=1}^m \ell^{(m-i)/2}(q^*)^{i/4+\eps}\sum_{b=1}^{\ell} b^{-(m-i)/2} \right),
%&= \ell^{m/2} L(\chi\psi,  \frac{m}{2})+O\left( \right)
\end{align*}
where $\chi=\chi_1\cdots \chi_m$ as above. This gives the wanted upon using 
$$\sum_{b=1}^{\ell-1} \psi(b)\chi(b) b^{-m/2}=L(\chi\psi,  m/2 )+O(\ell^{1-m/2}), $$
and Lemma \ref{lem:key} to bound the error-term.
\end{proof}
\begin{remark}\label{remark:m=2}
Using the exact same proof scheme as above one gets in the case $m=2$:
$$ \sum_{b=1}^{\ell-1} \psi(b) L\left(\chi_1, b/\ell, \frac{1}{2}\right)L\left(\chi_2, b/\ell, \frac{1}{2}\right)\ll_\eps (q^*)^{1/2+\eps}\ell^{1+\eps},  $$
which will suffice for our purposes (although one can probably do better using other techniques).
%$$ \ell L(\chi\psi,1)+O_\eps((q^*)^{1/2}\ell^{1/2}+(q^*)^{1/2+\eps}\ell),  $$
\end{remark}
With applications to non-vanishing in mind, we will obtain an asymptotic evaluation for the moments as in Theorem \ref{thm:generaltwist} in the case $\psi=1$ and $\chi_i=1$. This is exactly the moment in \cite[Theorem 3]{EgamiMatsu02}, but we obtain an estimate \emph{uniform} in $m$.
\begin{thm}\label{thm:uniforminm}
Let $m\geq 3$. Then we have
\begin{align}
\label{eq:asympexptwist}\sum_{b=1}^{\ell-1} \zeta(0, b/\ell, 1/2)^m =  \ell^{m/2} \zeta(m/2)+O(\ell^{m/2-1/2}+\ell \log \ell), %+c_{a,m}\, (\log \ell)\ell
\end{align}  
as $\ell\rightarrow \infty$, uniformly in $m,\ell$. 
\end{thm}
\begin{proof}
Let $C$ be such that $|\zeta(0, c, 1/2)- c^{-1/2}|\leq C$ uniformly for $c\in (0,1]$. We deduce for $\ell$ large enough
\begin{align*}
&\sum_{b=1}^{\ell-1} \zeta(0, b/\ell, 1/2)^m \\
&= \sum_{b=1}^{\ell-1} (b/\ell)^{-m/2}+O\left( (b/\ell)^{-(m-1)/2}+\ldots + C^{m-1}   \right)\\
&=\ell^{m/2} \zeta(m/2)+O( \ell^{m/2-1/2}+C\ell^{m/2-1}+\ldots+C^{m-3}\ell^{3/2}+C^{m-2}\ell \log \ell+\ell C^{m-1}   )\\
&=\ell^{m/2} \zeta(m/2)+O\left(\ell^{m/2-1/2} \left(\sum_{k=0}^{m-1} (C/\ell^{1/2})^k+ C^m \ell \log \ell   \right)  \right)\\
&=\ell^{m/2} \zeta(m/2)+ O(\ell^{m/2-1/2}+C^{m-2} \ell \log \ell)
\end{align*}
using the relation $a^n-b^n=(a-b)(a^{n-1}+\ldots + b^{n-1})$ in the first equality and using Lemma \ref{lem:key} to bound the error-terms. This yields the wanted (note that the error-term $\ell \log \ell$ is only relevant when $m=3$).
\end{proof}

\subsection{Moments of the twisted periodic zeta function}
We will now turn to the case of varying the $\alpha$-direction. In this case there is a dicotomy in the asymptotic behavior of the moments of twisted period zeta functions according to whether the corresponding moduli are co-prime or not.  
\begin{thm}\label{thm:nm}
Let $\psi_1,\ldots , \psi_m$ be primitive Dirichlet characters and let $\ell_i\geq 1$ denote  the modulus of $\psi_i$.
\begin{enumerate}
\item Assume that $\ell_1,\ldots, \ell_m$ are pairwise coprime. Then the function 
$$ \alpha \mapsto \prod_{i=1}^m L(\alpha,\psi_i, 1/2),$$
belongs to $L^{2-\eps}([0,1])$ for all $\eps>0$ and for primes $q$ we have
\begin{align}
\label{eq:twistperiod}\frac{1}{\varphi(q)}\sum_{a=1}^{q-1} \prod_{i=1}^m L(a/q, \psi_i, 1/2)= \int_{0}^1\prod_{i=1}^m L(\alpha,\psi_i, 1/2) d\alpha+O_{\psi_i}(q^{-1/2}),
\end{align}
where the implied constant is allowed to depend on $\psi_1,\ldots, \psi_m$.

\item Let $m\geq 2$ and let $\ell_i=\ell$ be independent of $1\leq i\leq m$. Then for primes $q$ we have 
\begin{align}
\label{eq:twistperiod2}\sum_{a=1}^{q-1} \prod_{i=1}^m L(a/q, \psi_i, 1/2)= c(\psi_1,\ldots, \psi_m) q^{m/2}+O_{\ell,m}(q^{m/2-1/2}),
\end{align}
 where 
\begin{equation}\label{eq:cpsi}c(\psi_1,\ldots, \psi_m):=\left(\psi(-1)\left(\tfrac{1+i}{\sqrt{2}}\right)^m+\left(\tfrac{1-i}{\sqrt{2}}\right)^m\right)\frac{\prod_{i=1}^m \tau(\psi_i)}{(2\ell)^{m/2}} \psi(q)L( \overline{\psi},m/2),
\end{equation}
with $\psi:=\prod_{i=1}^m \psi_i$ and we put $c(\psi_1,\overline{\psi_1})=0$.% The main term is non-trivial exactly if $$i^m\neq -\psi(-1).$$
\end{enumerate}
\end{thm}
\begin{proof}
Put 
$$f(\alpha):=\prod_{i=1}^m L(\alpha, \psi_i, 1/2), $$
for $\alpha\in \R$ and $\alpha\notin \{\tfrac{b}{\ell_i}:b\in \Z,i=1,\ldots,m\}$. By Proposition \ref{prop:analyticprop} we have that $f(\alpha)$ is bounded as $\alpha\rightarrow 0$ and that
$$ |f(\alpha)|\ll_{\ell_i} \frac{1}{\prod_{i=1}^m|\!|\ell_i\alpha |\!|^{1/2}}, $$
where $|\!|x|\!|=\min(\{x\},1-\{x\})$ denotes the distance to the nearest integer.

(1): Assume that the $\ell_i$'s are pairwise coprime. For $0\leq \alpha<\max_j\tfrac{1}{2\ell_j}$ and  $ 1-\max_j\tfrac{1}{2\ell_j}<\alpha\leq1$ we get from Proposition \ref{prop:analyticprop} that $f$ is bounded. Now consider the case where $\max_j\tfrac{1}{2\ell_j}\leq \alpha\leq 1-\max_j\tfrac{1}{2\ell_j}$ so that $\alpha$ is bounded away from $0$ and $1$. Then for any $1\leq j\leq m$ it holds that $|\!|\ell_j\alpha|\!|\leq \eps$ with $0<\eps<\tfrac{1}{2\ell_j}$ exactly if
$$\alpha\in\left[\tfrac{b_j}{\ell_j}-\eps,\tfrac{b_j}{\ell_j}+\eps\right], $$
for some $1\leq b_j\leq \ell_j-1$. By the coprimality assumption we have
\begin{equation}\label{eq:gaps} \left|\frac{b_j}{\ell_j}-\frac{b_i}{\ell_i}\right|\geq \frac{1}{\ell_i\ell_j}, \end{equation}
for any $1\leq b_j\leq \ell_j-1$ and $1\leq b_i\leq \ell_i-1$.  Thus we conclude that in fact the following stronger bound holds:
\begin{equation}\label{eq:boundfalpha}|f(\alpha)|\ll_{\ell_i} \max_{i=1,\ldots, m}|\!|\ell_i\alpha |\!|^{-1/2}.\end{equation}
This shows that $f$ does indeed belong to $L^{2-\eps}([0,1])$ for all $\eps>0$. We will now bound the derivative of $f$ as this will be needed when applying summation formulas. For $\Re s>3/2$ we get by differentiating term-by-term
$$\frac{\partial}{\partial \alpha} \zeta(\alpha, c , s)= 2\pi i (\zeta(\alpha, c , s-1)-c\cdot \zeta(\alpha, c , s)),\quad c>0, \alpha\notin \Z,$$
and by uniqueness of analytic continuation this equation holds at $s=1/2$ as well. By the functional equation (see \cite[Thm. 3.2]{LauGar02} or (\ref{eq:FE}) above) we get for $0<\alpha<1$:
\begin{align}
\zeta(\alpha, c , -1/2)=\frac{\Gamma(3/2)}{(2\pi)^{3/2}}\left(e(\tfrac{3}{8}-\alpha c)\zeta(-c,\alpha,3/2)+e(-\tfrac{3}{8}+(1-\alpha) c)\zeta(c,1-\alpha,3/2)\right).
\end{align}   
Using the standard bound $\zeta(\alpha, c , \sigma)\ll_\sigma c^{-\sigma}$ for $\sigma>1$ (which follows directly from the definition (\ref{eq:defLerch})), we conclude that
$$\frac{\partial}{\partial \alpha} \zeta(\alpha, c , 1/2)\ll |\!|\alpha|\!|^{-3/2}.$$
Plugging this into the formula (\ref{eq:MCtwistedperiodic}) for the twisted period zeta function we get 
$$ \frac{\partial}{\partial \alpha} L(\alpha,\psi_i,1/2)\ll_{\ell_i} |\!|\ell_i\alpha |\!|^{-3/2}$$
By the chain rule and the disjointness property (\ref{eq:gaps}) of the poles of $L(\alpha,\psi_i,1/2),i=1,\ldots ,m$, we get for $\alpha$ bounded away from $0$:
\begin{equation}\label{eq:derivbnd}f^\prime(\alpha)\ll_{\ell_i} \frac{1}{\prod_i|\!|\ell_i\alpha |\!|^{3/2}}\ll_{\ell_i}\max_i|\!|\ell_i\alpha |\!|^{-3/2}.\end{equation}
As $\alpha\rightarrow 0$, we similarly get  
$$f^\prime(\alpha)=\sum_{i=1}^m \frac{\partial}{\partial \alpha} L(\alpha,\psi_i,1/2)\prod_{j\neq i}L(\alpha,\psi_j,1/2)\ll_{\ell_i}\max_i|\!|\ell_i\alpha |\!|^{-3/2},  $$
using here also the boundedness (\ref{eq:twistperiodic2}) of $L(\alpha,\psi_i,1/2)$ as $\alpha\rightarrow 0$. 

We may assume that $q>\prod_{i=1}^m \ell_i$. We will now split up the sum in question into subsegments where $f$ is smooth so that we can apply the Euler--Maclaurin summation formula. With this in mind, let 
$$0=a_0<a_1<\ldots<a_{k}<a_{k+1}=q-1,$$ 
be such that for each $1\leq i\leq m$ and $1\leq b<\ell_i$ there exists a unique $1\leq j\leq k$ such that 
$$\tfrac{a_{j}}{q}<\tfrac{b}{\ell_i}<\tfrac{a_j+1}{q}.$$
Note that $k=\sum_{i=1}^m (\ell_i-1)$. We divide the sum as follows:
 \begin{align*}
\sum_{a=1}^{q-1} f(\tfrac{a}{q})&= \sum_{i=0}^{k} \sum_{n=a_{i}+1}^{a_{i+1}} f(\tfrac{n}{q})
\end{align*}
We note that by construction $f$ is smooth on each of the intervals $[a_{i}+1,a_{i+1}]$. By applying the Euler--Maclaurin formula \cite[Thm. 4.2]{IwKo} to each of the inner sums we obtain
\begin{align*}
\sum_{n=a_{i}+1}^{a_{i+1}} f(\tfrac{n}{q})&=\int_{a_{i}+1}^{a_{i+1}} f(\tfrac{x}{q})dx+ \frac{1}{2}\left(f\left(\tfrac{a_{i}+1}{q}\right)+f\left(\tfrac{a_{i+1}}{q}\right)\right)+O\left( \int_{a_{i}+1}^{a_{i+1}} q^{-1}|f^\prime(\tfrac{x}{q})|dx\right). \end{align*}
Since $|\tfrac{a}{q}-\tfrac{b}{\ell_i}|\geq \tfrac{1}{\ell_i q}$ for any $0\leq b<\ell_i, i=1,\ldots,m$ and $1< a<q$ (by coprimality), it follows from (\ref{eq:boundfalpha}) that
$$f\left(\tfrac{a_{i}+1}{q}\right), f\left(\tfrac{a_{i+1}}{q}\right) \ll_{\ell_i} q^{1/2}.$$
Combining the bound (\ref{eq:derivbnd}) and the inequalities characterizing the sequence $a_i$, we conclude that 
$$f^\prime(\tfrac{x}{q})\ll_{\ell_i} \left|x-\tfrac{a_i+1}{q}\right|^{-3/2}+\left|\tfrac{a_{i+1}}{q}-x\right|^{-3/2},\quad  x\in [a_{i}+1,a_{i+1}].$$ 
Inserting these bound gives: 
\begin{align*}
\sum_{a=1}^{q-1} f(\tfrac{a}{q})&= \sum_{i=0}^{k} \int_{a_{i}+1}^{a_{i+1}} f(\tfrac{x}{q})dx+O_{\ell_i}(q^{1/2}),
\end{align*}
using that $k$ depends only on $\ell_1,\ldots, \ell_m$. For each $1\leq i\leq k$, we have 
\begin{equation}\label{eq:exprbnd} \int_{a_{i}}^{a_{i}+1} f(\tfrac{x}{q})dx\ll_{\ell_i} \int_{a_{i}}^{a_{i}+1} |x-\tfrac{b}{\ell_j}|^{-1/2}dx \ll_{\ell_i} 1, \end{equation}
where $a_{i}/q<b/\ell_j<a_{i+1}/q$, and similarly we can bound the expression (\ref{eq:exprbnd}) for $i\in \{0,k+1\}$. This yields
\begin{align*}
\sum_{a=1}^{q-1} f(\tfrac{a}{q})= \int_0^{q}  f(\tfrac{x}{q})dx+O_{\ell_i}(q^{1/2})=q\int_0^1 f(x)dx+O_{\ell_i}(q^{1/2}),
\end{align*}
as wanted.

(2): Now assume that there exists $\ell\geq 1$ such that $\ell_i=\ell$ for all $i$ and that $m\geq 2$. By Proposition \ref{prop:analyticprop} we see that the sum $\sum_{a=1}^{q-1} f(\tfrac{a}{q})$ is equal to 
\begin{align*}
\sum_{a=1}^{q-1} \prod_{i=1}^m\left(\frac{\tau(\psi_i)\overline{\psi_i}\left(\lfloor \tfrac{\ell a}{q}\rfloor\right)(1+i)}{2\ell^{1/2}} \{\tfrac{\ell a}{q}\}^{-1/2}+\frac{\tau(\psi_i)\overline{\psi_i}\left(\lceil \tfrac{\ell a}{q}\rceil\right)(1-i)}{2\ell^{1/2}} (1-\{\tfrac{\ell a}{q}\})^{-1/2}+O_\ell(1)\right).
\end{align*}
Let $\overline{q},\overline{\ell}\in \Z$ be such that $q\overline{q}+\ell\overline{\ell}=1$. After the change of variable $a\mapsto \overline{\ell}a$ the above equals:
\begin{align*}
\sum_{a=1}^{q-1} \prod_{i=1}^m\left(\frac{\tau(\psi_i)\overline{\psi_i}\left(\lfloor \tfrac{\ell \overline{\ell} a}{q}\rfloor\right)(1+i)}{2\ell^{1/2}} (\tfrac{a}{q})^{-1/2}+\frac{\tau(\psi_i)\overline{\psi_i}\left(\lceil \tfrac{\ell \overline{\ell} a}{q}\rceil\right)(1-i)}{2\ell^{1/2}} (1-\tfrac{a}{q})^{-1/2}+O_\ell(1)\right).
\end{align*}
Observe that for any $a\in (\Z/q\Z)^\times$ and $1\leq i\leq m$ at least two out of three of the terms in the above product are $O_\ell(1)$. So if we expand the product we conclude that the contribution from the cross-terms is bounded by $ O_{\ell,m}(q^{m/2-1/2})$. We thus get two main terms; one corresponding to $(a/q)^{-m/2}$ and one corresponding to $(1-a/q)^{-m/2}$.   

To evaluate the contribution from the main terms, we start by noticing that for $1\leq a\leq q-1$: 
$$\lfloor \tfrac{\ell \overline{\ell} a}{q}\rfloor= \lfloor \tfrac{(1-q\overline{q}) a}{q}\rfloor=\lfloor \tfrac{a}{q}-a\overline{q}\rfloor=-a\overline{q},$$
and similarly $\lceil \tfrac{\ell \overline{\ell} (q-a)}{q}\rceil=a\overline{q}+\ell\overline{\ell}$. With $\psi:=\prod_{i=1}^m \psi_i$, the first main term equals:
\begin{align*}
\sum_{a=1}^{q-1} \prod_{i=1}^m\left(\frac{\tau(\psi_i)\overline{\psi_i}(\lfloor \tfrac{\ell \overline{\ell} a}{q}\rfloor)(1+i)}{2\ell^{1/2}} (\tfrac{a}{q})^{-1/2}\right)&= q^{m/2}\left(\frac{1+i}{2\ell^{1/2}}\right)^{m} \psi(-q)\left(\prod_{i=1}^m \tau(\psi_i)\right)\left(\sum_{a=1}^{q-1} \frac{\overline{\psi}(a)}{a^{m/2}}\right)  
%\\&=\begin{cases} q^{m/2}\left(\frac{1+i}{2\ell^{1/2}}\right)^{m}\psi(-q)L(\overline{\psi},m/2)\prod_i \tau(\psi_i)+O_\ell(q^{m/2-1/2}),&(m,\psi)\neq (2,\mathbf{1}) \\q\left(\frac{1+i}{2\ell^{1/2}}\right)^{m}\psi(-q)(\log q+\gamma)\prod_i \tau(\psi_i)+O_\ell(1),&m=2, \psi=\mathbf{1}\end{cases}
.\end{align*}
By estimating the tail (using partial summation in the latter case) we get   
$$\sum_{a=1}^{q-1} \frac{\overline{\psi}(a)}{a^{m/2}}=\begin{cases} 
L(\overline{\psi},m/2)+O_\ell(q^{1-m/2}),& m\geq 3\\
L(\overline{\psi},1)+O_\ell(q^{-1}),& m=2, \psi\neq \mathbf{1},
%\log q+\gamma +O(q^{-1}),& m=2, \psi= \mathbf{1}\modulo \ell 
\end{cases} $$
and finally in the case $m=2, \psi= (\mathbf{1}\modulo \ell)$, we conclude by the principle of inclusion and exclusion, the classical formula $\sum_{n=1}^N1/n=\log N+\gamma+O(N^{-1})$ and $\log \lfloor (q-1)/d \rfloor=\log q/d+O_d(q^{-1})$ that:
\begin{align}
\sum_{a=1}^{q-1} \frac{\overline{\psi}(a)}{a}=\sum_{\substack{1\leq a\leq q-1\\ (a,\ell)=1}} \frac{1}{a}=\sum_{d|\ell}\mu(d) \sum_{a=1}^{\lfloor (q-1)/d \rfloor} \frac{1}{ad}=\sum_{d|\ell}\frac{\mu(d)}{d}(\log q-\log d+\gamma)+O_\ell(q^{-1}).
\end{align}
The second main term is given by the same expressions just with $(\frac{1-i}{2\ell^{1/2}})^m$ in place of $(\frac{1+i}{2\ell^{1/2}})^m$ and  $\psi(q)$ in place of $\psi(-q)$. Summing up the two  yields the required formula. The last claim of the proposition follows from:
\begin{align*}\psi(-1)\left(\tfrac{1+i}{\sqrt{2}}\right)^m+\left(\tfrac{1-i}{\sqrt{2}}\right)^m=
\left(\tfrac{1+i}{\sqrt{2}}\right)^m \left(\psi(-1)+\left(\tfrac{1-i}{\sqrt{2}}\right)^{2m}\right)=\left(\tfrac{1+i}{\sqrt{2}}\right)^m (\psi(-1)+(-i)^m). 
 \end{align*}
 Note in particular that in the case $m=2, \psi= (\mathbf{1}\modulo \ell)$ the two main terms cancel. 
\end{proof}
\subsection{Double moments of the Lerch zeta function}
Finally we calculate a double average of the Lerch zeta function.%, which also grows quite slowly.
\begin{thm}\label{thm:bothdirections} Let $q,\ell\geq 2$ be two integers. Then we have for $m\geq 1$
%$$\sum_{a=1}^{q-1} \sum_{b=1}^{\ell-1} |\zeta(a/q, b/\ell, 1/2)|^{2n}=\frac{\zeta(n)}{2^n} q^{n}\ell +\zeta(n)q \ell^{n}+\sum_{k=2}^{n-1}c_{n,k}q^{n+1-k}\ell^k +O_n(\max(q,\ell)^{n+1/2})). $$
$$\sum_{a=1}^{q-1} \sum_{b=1}^{\ell-1} |\zeta(a/q, b/\ell, 1/2)|^{2m}=\frac{\zeta(m)}{2^{m-1}} q^{m}(\ell-1) +\zeta(m)(q-1) \ell^{m}+O_m(q^{m-1/2}\ell+ q\ell^{m-1/2}). $$%\zeta(n)q\ell(2^nq^{n-1}+\ell^{n-1}+O_n(\max(q,\ell)^{n-3/2})).  $$
\end{thm}
\begin{proof} 
Using (\ref{eq:expinc2}), we see that for $\alpha,c\in (0,1)$:
$$ \zeta(\alpha, c, 1/2)= \frac{1+i}{2}\alpha^{-1/2}+\frac{1-i}{2}(1-\alpha)^{-1/2}+c^{-1/2} +O(1).$$
Thus we get by expanding:
\begin{align*}
&\sum_{a=1}^{q-1} \sum_{b=1}^{\ell-1} |\zeta(a/q, b/\ell, 1/2)|^{2m} \\
&= \sum_{a=1}^{\lfloor q/2\rfloor} \sum_{b=1}^{\ell-1} \left(2^{-m}(a/q)^{-m}+(b/\ell)^{-m}\right)+\sum_{a=\lfloor q/2\rfloor+1}^{q-1} \sum_{b=1}^{\ell-1} \left(2^{-m}(1-a/q)^{-m}+(b/\ell)^{-m}\right)\\
&+ O_m\left(\sum_{k=1}^{2m-1}\sum_{a=1}^{q-1}\sum_{b=1}^{\ell-1}(a/q)^{-k/2}(b/\ell)^{-(2m-k)/2} \right).
\end{align*}
Now the theorem follows by applying Lemma \ref{lem:key}.
\end{proof}
\section{Applications to wide moments of Dirichlet $L$-functions and non-vanishing}
In this section we will combine the moment calculations from the previous section with the connection to Dirichlet $L$-functions, as in Proposition \ref{prop:BirchStevens},  to calculate a number of different wide moments. Then we will combine these wide moments calculation with some additional analytic input (subconvexity and second moment calculations) to obtain \emph{weak simultaneous non-vanishing} results. 
\subsection{Wide moments}\label{sec:widemoment}
First of all we will prove the asymptotic expansion as in Theorem \ref{thm:asympexpintro} extending results of Heath-Brown \cite{HeathBrown81} and Egami--Matsumoto \cite{EgamiMatsu02}.
\begin{proof}[Proof of Theorem \ref{thm:asympexpintro}]
By Proposition \ref{prop:BirchStevens} and (\ref{eq:periodformula1}), we see that for for $(b,\ell)=1$, we have
$$ \widehat{L(\chi_i,\cdot)}(\psi)=\frac{1}{\varphi(\ell)}\sum_{b \modulo \ell} L(\chi_i,b/\ell,1/2)\overline{\psi}(b)= \frac{\ell^{1/2}}{\varphi(\ell)}L(\chi_i\overline{\psi},1/2). $$
%$$ L(\chi_i, b/\ell, 1/2)=\frac{\ell^{1/2}}{\varphi(\ell)}\sum_{\psi \modulo \ell} L(\chi_i\psi,1/2)\overline{\psi}(b). $$
Thus by the Fourier identity (\ref{eq:fouriertrick}), we conclude that the lefthand side of (\ref{eq:widemomentmain}) is equal to
$$ \ell^{-m/2}\sum_{\substack{1\leq b<\ell\\ (b,\ell)=1}} \prod_{i=1}^m L(\chi_i, b/\ell,1/2). $$
Now using that for any function $f:\R\rightarrow \C$, we have by M\"{o}bius inversion;
$$ \sum_{\substack{1\leq b<\ell,\\ (b,\ell)=1}} f(b/\ell)
= \sum_{d|\ell} \mu(\ell/d) \sum_{b=1}^{d-1} f(b/d), $$
the claimed result follows from Theorem \ref{thm:asympexp}.
\end{proof}
Similarly, we get an asymptotic formula for a general $\psi$-twisted wide moment, with the further complication that we will restrict to only \emph{primitive} Dirichlet character. %This wil with applications to non-vanishing in mind. 
\begin{cor}\label{cor:uniforminq} For $i=1,\ldots, m$ with $m\geq 3$, let $\chi_i$ modulo $q_i$ be a Dirichlet character and put $q^*= \max(q_1,\ldots, q_m)$. Then we have for $\ell\gg (q^*)^{m/(2m-4)+\eps}$ prime and $\psi \modulo \ell$:
\begin{align}\label{eq:primitivemoment1}
&\frac{1}{(\ell-2)^{m-1}}\sideset{}{^*}\sum_{\substack{\psi_1,\ldots, \psi_m \, (\ell)\\ \psi_1\cdots \psi_m=\overline{\psi}}} \prod_{i=1}^m L(\chi_i\psi_i,1/2)\\
\nonumber&= L(\chi\psi,  m/2)+O_{m,\eps}((q^*)^{m/4+\eps}\ell^{1-m/2}+(q^*)^{1/4+\eps}\ell^{-1/2}),      
\end{align}
where $\chi=\chi_1\cdots \chi_m$. Here the sum is restricted to primitive Dirichlet characters $\psi_1,\ldots, \psi_m$. 
\end{cor}
\begin{proof} Using the principle of inclusion and exclusion, we see that the lefthand side of (\ref{eq:primitivemoment1}) is equal to
\begin{align*}
\frac{1}{(\ell-2)^{m-1}} \sum_{I\subset\{1,\ldots, m\}} (-1)^{|I|}\prod_{i\in I} L(\chi_i\psi_0,1/2) \sum_{\substack{\psi_i\, (\ell), i \notin I\\ \Pi_{i\notin I}\psi_i=\psi}} \prod_{i\notin I} L(\chi_i\psi_i,1/2) ,
\end{align*}
where $\psi_0$ denotes the principal Dirichlet character modulo $\ell$. Combining the Fourier identity (\ref{eq:fouriertrick}), the calculation of Fourier transforms as in Proposition \ref{prop:BirchStevens} and the moment calculation in Theorem \ref{thm:generaltwist}, we can evaluate each of the inner sums above. The main term comes from $I=\emptyset$ with the error-term $O_{m, \eps}((q^*)^{m/4+\eps}\ell^{1-m/2}+(q^*)^{1/4+\eps}\ell^{-1/2})$. For $I\neq \emptyset$, we have the following bounds:
\begin{align*}&\ll_\eps \begin{cases}\ell^{m-|I|-1} L(\psi \prod_{i\notin I} \chi_i,(m-|I|)/2)=O( \ell^{m-|I|-1}), &  |I|<m-2\\ \ell^{1+\eps} (q^*)^{1/2+\eps} ,&  |I|=m-2\\ L(\psi \prod_{i\notin I}\chi_i,1/2)=O( (\ell q^*)^{1/2}) , & |I|=m-1   \\ 1, & |I|=m. \end{cases}
%&\ll \begin{cases}\ell^{(m-|I|)/2}, &  |I|<m-2\\ \ell\log \ell , & |I|=m-2\\ (\ell q^*)^{1/2}\ell, & |I|=m-1   \\ 1, & |I|=m. \end{cases}, 
\end{align*}
applying respectively, Theorem \ref{thm:generaltwist} for $|I|<m-2$ (using here the assumption $\ell\gg (q^*)^{m/(2m-4)+\eps}$ to bound the error-term in Theorem \ref{thm:generaltwist}), Remark \ref{remark:m=2} for $|I|=m-2$ and the convexity bound for $|I|=m-1$.  Thus the combined contribution from $I\not= \emptyset$ can be bounded by
\begin{align*}
& \ll_{m,\eps} \frac{1}{\ell^{m-1}} \Biggr(\sum_{i=1}^{m-3} (q^*)^{i/2}\ell^{m-1-i}+\ell^{1+\eps}(q^*)^{1/2+\eps}(q^*)^{(m-2)/2}\\
&\qquad\qquad\qquad\qquad\qquad+(\ell q^*)^{1/2}(q^*)^{(m-1)/2}+(q^*)^{m/2} \Biggr),   
\end{align*}  
using here the convexity bound $L(\psi_0\chi_i,1/2)\ll (q^*)^{1/2}$ for $i\in I$. Now it is easy to check that the above is indeed $O_{m, \eps}((q^*)^{m/4+\eps}\ell^{1-m/2}+(q^*)^{1/4+\eps}\ell^{-1/2})$ as wanted. 
\end{proof}
Next, we will obtain an asymptotic formula uniform in $m$ for wide moments with trivial twists. Again we will use the principle of inclusion and exclusion to restrict the sum to primitive Dirichlet characters.
\begin{cor}\label{cor:uniforminm}
Let $m\geq 3$ and let $\ell$ be prime. Then we have
\begin{align}\label{eq:primitivemoment}
\frac{1}{(\ell-2)^{m-1}}\sideset{}{^*}\sum_{\substack{\psi_1,\ldots, \psi_m \, (\ell)\\ \psi_1\cdots \psi_m=1}} \prod_{i=1}^m L(\psi_i,1/2)= \zeta(m/2)+O(\ell^{-1/2}+\ell^{1-m/2} \log \ell+\ell^{-1}e^{O(m/\ell)}), 
\end{align}
as $\ell\rightarrow \infty$, uniformly in $m,\ell$. Here the sum is restricted to primitive Dirichlet characters $\psi_1,\ldots, \psi_m$ modulo $\ell$. 
\end{cor}
%\begin{align*}
%\frac{1}{\varphi^*(\ell)^{m-1}} \sum_{I\subset\{1,\ldots, m\}} (-1)^{|I|}\prod_{i\in I} L(\chi_i \psi_0,1/2) \sum_{\psi_i\, (\ell), i \notin I} \prod_{i\notin I} L(\chi_i\psi_i,1/2) ,
%\end{align*}
\begin{proof}
We proceed as above using the principle of inclusion and exclusion to write the lefthand side of (\ref{eq:primitivemoment}) as
\begin{align*}
\frac{1}{(\ell-2)^{m-1}} \sum_{I\subset\{1,\ldots, m\}} (-1)^{|I|}\prod_{i\in I} L(\psi_0,1/2) \sum_{\psi_i\, (\ell), i \notin I} \prod_{i\notin I} L(\psi_i,1/2) ,
\end{align*}
where $\psi_0$ denotes the principal character modulo $\ell$. By the Fourier identity (\ref{eq:fouriertrick}) and Proposition \ref{prop:BirchStevens}, we can apply Theorem \ref{thm:uniforminm} to each of the inner sums above. Using that $L(\psi_0,1/2)\leq C$ for some constant $C$ independent of $\ell$, we see that the lefthand side of (\ref{eq:primitivemoment}) is
\begin{align*}
&\zeta(m/2)+ O\left(\ell^{-1/2}+\ell^{1-m/2}\log \ell+\sum_{i=1}^m\binom{m}{i}(C/\ell)^{i}\right) \\
&=\zeta(m/2)+ O\left(\ell^{-1/2}+\ell^{1-m/2}\log\ell +((1+C\ell^{-1})^m-1)\right)\\
&= \zeta(m/2)+ O\left(\ell^{-1/2}+\ell^{1-m/2}\log\ell +\ell^{-1}e^{O(m/\ell)}\right),
\end{align*} 
as wanted.
\end{proof}
\subsubsection{Further examples of wide moments}
We now consider wide moments of Dirichlet $L$-functions but this time twisted by Gau{\ss} sums. This corresponds exactly to moments of twisted periodic zeta functions as in Section \ref{sec:momentsLerch}. We observe that in the case where the moduli are co-prime there is massive cancellation in the wide moments, which means that these results are not very useful for obtaining non-vanishing results. On the other hand, it is usually very hard to obtain such a huge amount of cancellation using an approximate functional approach. When the moduli are equal one gets the same asymptotic behavior as in Corollary \ref{cor:uniforminq}.   
\begin{cor}\label{cor:nm}Let $\psi_1,\ldots , \psi_m$ be primitive Dirichlet characters and let $\ell_i\geq 1$ be the modulus of $\psi_i$.
\begin{enumerate}
\item Assume that $m\geq 3$ and that $\ell_1,\ldots, \ell_m$ are pairwise coprime. Then for $q$ prime it holds that
$$\frac{1}{(q-2)^{m-1}}\sideset{}{^\ast}\sum_{\substack{\chi_1,\ldots, \chi_m(q)\\ \chi_1\cdots \chi_m=\mathbf{1}}} \prod_{i=1}^m \tau(\overline{\chi_i})L(\chi_i\psi_i,1/2)=\int_0^1 \prod_{i=1}^mL(\alpha,\psi_i,1/2) d\alpha+O_{\ell_i}(q^{-1/2}).$$
Here the sum is restricted to primitive Dirichlet characters $\chi_1,\ldots, \chi_m$.
\item Assume that $m\geq 2$ and that there exists $\ell\geq 1$ such that $\ell_i=\ell$ for all $i=1,\ldots, m$. Then for $q$ prime it holds that
\begin{align}
\label{eq:nm}\frac{1}{(q-2)^{m-1}}\sideset{}{^*}\sum_{\substack{\chi_1,\ldots, \chi_{m}\, (q)\\ \chi_1\cdots \chi_{m}=1}}\prod_{i=1}^{m}\epsilon(\chi_i) L(\psi_i\chi_{i},1/2)  = c(\psi_1,\ldots, \psi_m)+O_{\ell,m}(q^{-1/2}),\end{align}
Here the sum is restricted to primitive Dirichlet characters $\chi_1,\ldots, \chi_m$, $c(\psi_1,\ldots, \psi_m)$ is given by the formula (\ref{eq:cpsi}), and $\epsilon(\chi):=\tau(\overline{\chi})/q^{1/2}$.
\end{enumerate}
\end{cor}
%Let $m\geq 2$ and $\ell\geq 2$ be integers. For $i=1,\ldots, m$, let $\psi_i$ be a primitive Dirichlet characters modulo $\ell$. 
%\end{align}
%where the sum is restricted to primitive characters modulo $q$, $\varphi^\ast(q):=q-2$ denotes the number of primitive characters modulo $q$, and $\epsilon(\chi):=\tau(\overline{\chi})/q^{1/2}$.
%Fix $m\geq 3$. Let $\psi_1,\ldots , \psi_m$ be primitive Dirichlet characters and let $\ell_i\geq 1$ be the modulus of $\psi_i$. Assume that $\ell_1,\ldots, \ell_m$ are pairwise coprime. Then for $q$ prime it holds that
%$$\frac{1}{\varphi^\ast(q)^{m-1}}\sideset{}{^\ast}\sum_{\substack{\chi_1,\ldots, \chi_m(q)\\ \chi_1\cdots \chi_m=\mathbf{1}}} \prod_{i=1}^m \tau(\overline{\chi_i})L(\chi_i\psi_i,1/2)=\int_0^1 \prod_{i=1}^mL(\alpha,\psi_i,1/2) d\alpha+O_{\ell_i}(q^{-1/2}).$$
\begin{proof}
The proof follows as above combining the Fourier identity (\ref{eq:fouriertrick}), Proposition \ref{prop:BirchStevens}, the principle of inclusion and exclusion, the bound  
$$\nu_{1/2}(\chi^*, 1, q/q^*)\tau(\chi^*)\ll q^{1/2},\quad \text{for Dirichlet characters }\chi\modulo q,$$
(recall the notation from Proposition \ref{prop:BirchStevens}) and finally Theorem \ref{thm:nm}. We will skip the details.
\end{proof}
Finally, we obtain a calculation of the following double average wide moment.
\begin{cor}\label{cor:double}
Let $q,\ell$ be primes and $m\geq 2$. Then we have:
\begin{align*} 
&\frac{\ell^m}{(\varphi(q)\varphi(\ell))^{2m-1}}\sideset{}{^*}\sum_{\substack{\chi_1,\ldots, \chi_{2m}\, (q),\, \psi_1,\ldots, \psi_{2m}\, (\ell)\\  \Pi_{i=1}^m \chi_i\overline{\chi_{m+i}}=\Pi_{i=1}^m \psi_i\overline{\psi_{m+i}}=1}} \prod_{i,j=1}^{2m}\tau(\overline{\chi_i})L(\chi_i \psi_j,1/2)\\
%=\zeta(n)q\ell^{1-n}(2^nq^{n-1}+\ell^{n-1}+O_n(\max(q,\ell)^{n-3/2}))   
&=\frac{\zeta(m)}{2^{m-1}} q^{m}(\ell-1) +\zeta(m)(q-1) \ell^{m}+O_m(q^{m-1/2}\ell+q\ell^{m-1/2})),   \end{align*}
as $q+\ell \rightarrow \infty$. Here the sum is restricted to primitive Dirichlet characters $\chi_1,\ldots, \chi_{2n}$ and $\psi_1,\ldots, \psi_{2n}$,
\end{cor}
\begin{proof}
The result follows using the same proof scheme as in Corollary \ref{cor:nm} using now Theorem \ref{thm:bothdirections}.
\end{proof}
\subsection{Non-vanishing}
In this final section we will use the wide moment calculations in Corollaries \ref{cor:uniforminq} and \ref{cor:uniforminm} combined with certain additional analytic input to obtain the non-vanishing results in Subsection \ref{sec:state} (notice that in the introduction we have interchanged $q\leftrightarrow \ell$ and $\chi\leftrightarrow \psi $ in the notation). Similarly, one could get non-vanishing results using Corollaries \ref{cor:nm} and \ref{cor:double}, but the results one obtains this way are not that interesting.
\begin{proof}[Proof of Theorem \ref{cor:nonvaniPY}] We see that in the range $q\gg (\ell^*)^{m/(2m-4)+\eps}$, we have from Corollary \ref{cor:uniforminq} that
$$\sideset{}{^*}\sum_{\substack{\chi_1,\ldots, \chi_m \, (q)\\ \chi_1\cdots \chi_m=\chi}} \prod_{i=1}^m L(\psi_i\chi_i,1/2) \gg_m  q^{m-1},$$
as $q\rightarrow \infty$. Using the subconvexity bound  
$$ L(\psi_i\chi_i,1/2)\ll_\eps (q\ell^*)^{1/6+\eps}, $$
due to Petrow--Young \cite{PetrowYoung19} we conclude
\begin{align*}
 (q\ell^*)^{m/6+\eps} \sideset{}{^*}\sum_{\substack{\chi_1,\ldots, \chi_m \, (q)\\ \chi_1\cdots \chi_m=\chi\\\Pi_{i=1}^m L(\psi_i\chi_i,1/2)\neq 0 }}\, 1 \gg_m  q^{m-1} ,
\end{align*}  
which yields the wanted.
\end{proof}
Now we restrict to the case of trivial twists in Theorem \ref{cor:nonvaniPY} (i.e. $\psi_i=1$ for $i=1,\ldots, m$ and $\chi=1$). Then we can do even better due to the following second moment calculation of Bettin \cite{Be17}; we have uniformly for $3 \leq m=o(q^{1/2} \log q)$ that   
\begin{equation}\label{eq:Bettin17} \frac{1}{(q-2)^{m-1}}\sideset{}{^*}\sum_{\substack{\chi_1,\ldots, \chi_m \, (q),\\ \chi_1\cdots \chi_m=1}} \prod_{i=1}^m|L(\chi_i, 1/2)|^2\sim \frac{\zeta(m/2)^2}{\zeta(m)}(\log (q/8\pi)+\gamma)^m,  \end{equation}
as $q\rightarrow \infty$. Here $\gamma$ denotes the Euler--Mascheroni constant and the sum is restricted to primitive Dirichlet characters $\chi_1,\ldots, \chi_m$.

\begin{proof}[Proof of Theorem \ref{cor:nonvaniBettin}]
Recall the notation of $\wide^*(q, m;1)$ in (\ref{eq:wide*}). We have by Cauchy--Schwarz:
\begin{align*}&\left|\sum_{(\chi_i)\in \wide^*(q, m;1)} \prod_{i=1}^m L(\chi_i,1/2)\right| \\
\leq &\left( \sum_{\substack{(\chi_i)\in \wide^*(q, m;1),\\ L(\chi_1,1/2)\cdots L(\chi_m,1/2)\neq 0}} 1 \right)^{1/2}\left(\sum_{(\chi_i)\in \wide^*(q, m;1)} \prod_{i=1}^m |L(\chi_i,1/2)|^2 \right)^{1/2},\end{align*}
which implies that 
\begin{align*}
 &\#\{(\chi_i)\in \wide^*(q, m;1): L(\chi_1,1/2)\cdots L(\chi_m,1/2)\neq 0  \}\\
 &\geq \frac{\left|\sum_{(\chi_i)\in \wide^*(q, m;1)} \prod_{i=1}^m L(\chi_i,1/2)\right|^2}{\sum_{(\chi_i)\in \wide^*(q, m;1)} \prod_{i=1}^m |L(\chi_i,1/2)|^2}\\
 &\gg \frac{q^{m-1}}{(\log q)^m},
\end{align*}
uniformly in $m\geq 3$ with $m=o(q^{1/2}\log q)$, using (\ref{eq:Bettin17}) and Corollary \ref{cor:uniforminm}. Here we are using:
$$ \frac{\zeta(m/2)^2}{\zeta(m)} \ll 1, $$
and 
$$ (\log (q/8\pi)+\gamma)^m = (\log q+\gamma-\log 8\pi )^m \leq (\log q)^m,$$
uniformly in $m$ and $q$.
\end{proof}
\bibliography{mybib}

\begin{thebibliography}{10}

\bibitem{Be17}
Sandro Bettin.
\newblock High moments of the estermann function.
\newblock {\em Algebra and Number Theory}, 13, 01 2017.

\bibitem{BeDr19}
Sandro Bettin and Sary Drappeau.
\newblock Limit laws for rational continued fractions and value distribution of
  quantum modular forms.
\newblock {\em Proceedings of the London Mathematical Society},
  125(6):1377--1425, 2022.

\bibitem{Chinta02}
Gautam Chinta.
\newblock Analytic ranks of elliptic curves over cyclotomic fields.
\newblock {\em J. Reine Angew. Math.}, 544:13--24, 2002.

\bibitem{ConstNor22}
Petru Constantinescu and Asbj{\o}rn Nordentoft.
\newblock Residual equidistribution of modular symbols and cohomology classes
  for quotients of hyperbolic $n$-space.
\newblock {\em Transactions of the American Mathematical Society}, 2022.

\bibitem{Davenport00}
Harold {Davenport}.
\newblock {\em {Multiplicative number theory. Revised and with a preface by
  Hugh L. Montgomery.}}, volume~74.
\newblock New York, NY: Springer, 2000.

\bibitem{EgamiMatsu02}
Shigeki {Egami} and Kohji {Matsumoto}.
\newblock {Asymptotic expansions of multiple zeta functions and power mean
  values of Hurwitz zeta functions}.
\newblock {\em {J. Lond. Math. Soc., II. Ser.}}, 66(1):41--60, 2002.

\bibitem{HeathBrown81}
D.~R. {Heath-Brown}.
\newblock {An asymptotic series for the mean value of Dirichlet
  \(L\)-functions}.
\newblock {\em {Comment. Math. Helv.}}, 56:148--161, 1981.

\bibitem{IwKo}
Henryk Iwaniec and Emmanuel Kowalski.
\newblock {\em Analytic number theory}, volume~53 of {\em American Mathematical
  Society Colloquium Publications}.
\newblock American Mathematical Society, Providence, RI, 2004.

\bibitem{KeaSna00}
J.~P. Keating and N.~C. Snaith.
\newblock Random matrix theory and {$L$}-functions at {$s=1/2$}.
\newblock {\em Comm. Math. Phys.}, 214(1):91--110, 2000.

\bibitem{KhanMilNgo19}
Rizwanur {Khan}, Djordje {Mili{\'c}evi{\'c}}, and Hieu~T. {Ngo}.
\newblock {Nonvanishing of Dirichlet L-functions, II}.
\newblock {\em arXiv e-prints}, page arXiv:1911.10268, November 2019.

\bibitem{LagLi12}
Jeffrey~C. {Lagarias} and Wen-Ching~Winnie {Li}.
\newblock {The Lerch zeta function. I: Zeta integrals}.
\newblock {\em {Forum Math.}}, 24(1):1--48, 2012.

\bibitem{LauGar02}
Antanas {Laurin\v{c}ikas} and Ramunas {Garunk\v{s}tis}.
\newblock {\em {The Lerch zeta-function}}.
\newblock Dordrecht: Kluwer Academic Publishers, 2002.

\bibitem{LeeSun19}
Jungwon {Lee} and Hae-Sang {Sun}.
\newblock {Dynamics of continued fractions and distribution of modular
  symbols}.
\newblock {\em arXiv e-prints}, page arXiv:1902.06277, Feb 2019.

\bibitem{LeopoldtKubota64}
Heinrich~W. Leopoldt and Tomio Kubota.
\newblock Eine p-adische theorie der zetawerte. teil i: Einf{\"u}hrung der
  p-adischen dirichletschen l-funktionen.
\newblock {\em Journal f{\"u}r die reine und angewandte Mathematik (Crelles
  Journal)}, 1964(214-215):328--339, 1964.

\bibitem{MaRu19}
Barry {Mazur} and Karl {Rubin}.
\newblock {Arithmetic conjectures suggested by the statistical behavior of
  modular symbols}.
\newblock {\em To appear in {E}xperimental {M}athematics}, page
  arXiv:1910.12798, October 2019.

\bibitem{Nordentoft20.4}
Asbj{\o}rn~Christian Nordentoft.
\newblock Central values of additive twists of cuspidal {$L$}-functions.
\newblock {\em Journal f{\"u}r die reine und angewandte Mathematik (Crelles
  Journal)}, 2021(776):255--293, 2021.

\bibitem{Nordentoft24}
Asbj{\o}rn~Christian Nordentoft.
\newblock Wide moments of {L}-functions. {I}: {Twists} by class group
  characters of imaginary quadratic fields.
\newblock {\em Algebra Number Theory}, 18(4):735--770, 2024.

\bibitem{Paley31}
R.~{Paley}.
\newblock {On the \(k\)-analogues of some theorems in the theory of the Riemann
  \(\zeta\)-function}.
\newblock {\em {Proc. Lond. Math. Soc. (2)}}, 32:273--311, 1931.

\bibitem{PeRi}
Yiannis~N. Petridis and Morten~S. Risager.
\newblock Arithmetic statistics of modular symbols.
\newblock {\em Invent. Math.}, 212(3):997--1053, 2018.

\bibitem{PetrowYoung19}
Ian {Petrow} and Matthew~P. {Young}.
\newblock {The fourth moment of Dirichlet $L$-functions along a coset and the
  Weyl bound}.
\newblock {\em arXiv e-prints}, page arXiv:1908.10346, August 2019.

\bibitem{Rohrlich89}
David~E. Rohrlich.
\newblock Nonvanishing of {$L$}-functions for {${\rm GL}(2)$}.
\newblock {\em Invent. Math.}, 97(2):381--403, 1989.

\bibitem{YoungAnnals11}
Matthew~P. {Young}.
\newblock {The fourth moment of Dirichlet \(L\)-functions}.
\newblock {\em {Ann. Math. (2)}}, 173(1):1--50, 2011.

\bibitem{Zach19}
Rapha\"el {Zacharias}.
\newblock {Non-annulation simultan\'ee de fonctions \(L\) de Dirichlet}.
\newblock {\em {Ann. Inst. Fourier}}, 69(4):1459--1524, 2019.

\bibitem{ZagierECM94}
Don {Zagier}.
\newblock {Values of zeta functions and their applications}.
\newblock In {\em First European congress of mathematics (ECM), Paris, France,
  July 6-10, 1992. Volume II: Invited lectures (Part 2)}, pages 497--512.
  Basel: Birkh\"auser, 1994.

\end{thebibliography}
\bibliographystyle{plain}
\end{document}